\documentclass[12pt]{amsart} 
\usepackage[left=3cm,top=2.5cm,bottom=2.5cm,right=3cm]{geometry}
\usepackage[oldstyle]{libertine}%
\usepackage{amssymb, amsmath, amsthm,bm}
\usepackage{mathtools}
\usepackage{graphicx,xspace}
\usepackage{epsfig}
\usepackage{enumitem}
\usepackage[usenames,dvipsnames]{xcolor}
\usepackage{tikz}

\usepackage{braket}
\usepackage[T1]{fontenc}
\usepackage[utf8]{inputenc}
\makeatletter
\renewcommand*\libertine@figurestyle{LF}
\makeatother
\usepackage[libertine,libaltvw,liby]{newtxmath}
\makeatletter
\renewcommand*\libertine@figurestyle{OsF}
\makeatother
\usepackage{subcaption}
\usepackage{yfonts}
\definecolor{green}{RGB}{0,127,0}
\definecolor{red}{RGB}{191,0,0}
\usepackage[colorlinks,citecolor=red,linkcolor=green,pagebackref=true]{hyperref}
\usepackage{todonotes}
\usepackage[capitalize]{cleveref}

\newtheorem{lemma}{Lemma}[section]
\newtheorem{theorem}[lemma]{Theorem}
\newtheorem{corollary}[lemma]{Corollary}
\newtheorem{proposition}[lemma]{Proposition}

\newtheorem{definition-lemma}[lemma]{Definition-Lemma}

\theoremstyle{definition}
\newtheorem{definition}[lemma]{Definition}
\newtheorem{remark}[lemma]{Remark}
\newtheorem{example}[lemma]{Example}

\newcommand{\pp}{\mathbf{p}}

\newcommand{\qq}{\mathbf{q}}

\newcommand{\C}{\mathbb{C}}
\newcommand{\Q}{\mathbb{Q}}
\newcommand{\R}{\mathbb{R}}

\newcommand{\J}{\mathsf{J}}

\newcommand{\al}{{(\alpha)}}
\newcommand{\bb}{\mathfrak{b}}

\newcommand{\M}{\mathcal{M}}
\newcommand{\tM}{\tilde{\mathcal{M}}}

\DeclareMathOperator{\Aut}{Aut}

\DeclareMathOperator{\wt}{wt}

\DeclareMathOperator*{\Res}{Res}
\DeclareMathOperator*{\Tr}{Tr}
\DeclareMathOperator{\SU}{SU}

\title[ $\bb$-Hurwitz numbers from refined topological recursion]{$\bb$-Hurwitz numbers from refined topological recursion}

	\author[N. K. Chidambaram]{Nitin Kumar Chidambaram}
\address{
	School of Mathematics, University of Edinburgh, James Clerk Maxwell Building, Peter Guthrie Tait Rd, Edinburgh EH9 3FD, U.K. }
	\address{
	\textit{	Present address:} Departamento de Matem\'aticas Fundamentales, UNED,
		Calle de Juan del Rosal, 28040 Madrid, Spain. 
}
\email{nitin.chidambaram@mat.uned.es}

\author[M.~Do\l\k{e}ga]{Maciej Do\l\k{e}ga}
\address{
Institute of Mathematics, 
Polish Academy of Sciences, 
ul. \'{S}niadeckich 8, 
00-956 Warszawa, Poland.
}
\email{mdolega@impan.pl}

\author[K.~Osuga]{Kento Osuga}
\address{Graduate School of Mathematics \& Kobayashi--Maskawa Institute for the Origin of Particles and the Universe, Nagoya University, Furocho, Nagoya, Aichi, 464-8601, Japan}
\address{Graduate School of Mathematical Sciences, University of Tokyo, Komaba 3-8-1, Meguro, Tokyo, 153-8914, Japan
}
\email{osuga@math.nagoya-u.ac.jp}

\thanks{NKC acknowledges the support of the ERC Starting Grant 948885, and the Royal Society University Research Fellowship. This research was funded in whole or in part by {\it Narodowe Centrum Nauki}, grant 2021/42/E/ST1/00162.  KO is supported by JSPS KAKENHI Grant number 22KJ0715 and 23K12968 (also in part by 21H04994 and 24K00525). This work is part of grant RYC2023-042878-I, funded by MCIN/AEI/10.13039/501100011033 and by the European Social Fund Plus (FSE+). For the purpose of Open Access, the authors have applied a CC-BY public copyright license to any Author Accepted Manuscript (AAM) version arising from this submission.}

\begin{document}
\emergencystretch 3em 
\begin{abstract}
We prove that single $G$-weighted $\bb$-Hurwitz numbers with internal faces are computed by refined topological recursion on a rational spectral curve, for certain rational weights $G$. Consequently, the $\bb$-Hurwitz generating function analytically continues to a rational curve. In particular, our results cover the cases of $\bb$-monotone Hurwitz numbers, and the enumeration of maps and bipartite maps (with internal faces) on non-oriented surfaces. As an application, we prove that the correlators of the Gaussian, Jacobi and Laguerre $\beta$-ensembles are computed by refined topological recursion. 
\end{abstract}

\maketitle
\tableofcontents

\section{Introduction}

 Motivated by topological string theory, Bouchard and Mari\~no \cite{BouchardMarino2008} conjectured  that simple Hurwitz numbers can be computed using the Chekhov--Eynard--Orantin topological recursion formalism \cite{ChekhovEynardOrantin2006,EynardOrantin2007}.   Subsequently proved in \cite{EynardMulaseSafnuk2011}, this conjecture  sparked an explosion of activity relating Hurwitz theory and topological recursion \cite{DoDyerMathews2017, DoLeighNorbury2016,  KramerPopolitovShadrin2022, Dunin-BarkowskiKramerPopolitovShadrin2019, BorotDoKarevLewanskiMoskovsky2023} culminating with the following recent result of \cite{AlexandrovChapuyEynardHarnad2018, BychkovDuninBarkowskiKazarianShadrin2024}: whenever the weight $G$ is a rational function times an exponential,  $G$-weighted Hurwitz numbers can be computed by topological recursion on an associated spectral curve.  The existence of topological recursion is a key  result that controls the structure of weighted Hurwitz numbers. Particular important consequences include the existence of a rational parametrization for their generating functions, existence of quantum curves and determinantal formulas, and ELSV-type formulas generalizing \cite{EkedahlLandoShapiroVainshtein2001}.

\subsection*{$\bb$-deformation}Motivated in part by the $\beta$-deformation of matrix models, one-parameter deformations of both Hurwitz theory and topological recursion have been 
 studied intensely in the last few years. Defined by \cite{ChapuyDolega2022}, $G$-weighted $\bb$-Hurwitz theory is a one-parameter interpolation between classical (complex) Hurwitz theory and its non-orientable (real) version. 
Given a formal power series $G(z) = \sum_{i\geq 0} g_i z^i$, the tau function of $G$-weighted  $\bb$-Hurwitz numbers, denoted $\tau^{(\bb)}_G$, is defined  as  
\begin{equation*}
	\tau_G^{{(\bb)}}:=\sum_{d \geq 0}\left(\frac{t \cdot \sqrt{\alpha}}{\hbar} \right)^d\sum_{\lambda \vdash d} 
	\check J_\lambda^{(\alpha)}(\sqrt{\alpha}\tilde{\pp})\prod_{\square
		\in \lambda}G(\hbar\cdot
	\tilde{c}_{\alpha}(\square)),
\end{equation*} where $\check J_\lambda^{(\alpha)}$ is an appropriately normalized Jack symmetric function and $\tilde c_\alpha$ is the deformed content of the box of the Young diagram (see \cref{sec:b-Hurwitz} for a detailed explanation). The tau function admits a topological expansion 
\begin{equation}
	\tau_G^{{(\bb)}}=
	\exp\left(\sum_{g\in\frac12\mathbb{Z}_{\geq0}}\sum_{n\in\mathbb{Z}_{\geq1}}\sum_{\mu_1,\ldots,\mu_n\in\mathbb{Z}_{\geq1}}\frac{\hbar^{2g-2+n}}{n!}\;
	F_{g,n} [\mu_1,\ldots,\mu_n] \;\frac{\tilde p_{\mu_1}}{\mu_1}\cdots
	\frac{\tilde p_{\mu_n}}{\mu_n}\right),\label{tau in terms of Fgn}
\end{equation}
and the  coefficients $F_{g,n}$ admit a combinatorial interpretation in terms of weighted real (non-oriented) branched coverings where the parameter $\bb =  \sqrt{\alpha}^{-1} - \sqrt{\alpha}$ measures non-orientability in a precise sense \cite{ChapuyDolega2022}. When $\bb = 0$, Jack functions reduce to Schur functions, only orientable contributions survive, and  $\tau^{(0)}_G$ becomes the generating function of $G$-weighted Hurwitz numbers, which is known to be a Kadomtsev--Petviashvili tau function~\cite{Guay-PaquetHarnad2017}.

Refined topological recursion, defined by \cite{KidwaiOsuga2023,Osuga2024a} building on  \cite{ChekhovEynard2006a}, is a one-parameter deformation of the  Chekhov--Eynard--Orantin topological recursion whose underlying (global) symmetry is the Virasoro algebra at  central charge $c = 1- 6 \bb^2$. Refined topological recursion takes as input a refined spectral curve, denoted $\mathcal S_{\bm \mu}$, which consists of  the  data $(\Sigma , x, y, \mathcal P_+, \{\mu_a\}_{a\in  \mathcal{P}_+})$ --- $\Sigma$ is a compact Riemann surface;  $x,y$ are two non-constant meromorphic functions on $\Sigma$ such that $x: \Sigma \to \mathbb P^1$ is of degree two;  $ \mathcal P_+$ is a subset of the set of the zeroes and poles of $y dx$ (possibly with some shift); and $\mu_a \in \mathbb C$ is  a complex number for every such $a \in \mathcal P_+$. 

The data of a refined spectral curve defines the unstable correlators $\omega_{0,1}, \omega_{0,2},$ and $\omega_{\frac{1}{2},1}$, and the refined topological recursion is a recursive (on the integer $(2g-2+n)$) construction of the stable correlators $\omega_{g,n} $ for $(g,n) \in \frac{1}{2} \mathbb Z_{\geq 0} \times \mathbb Z_{\geq 1}$ such that $(2g-2+n)>0$. When $\Sigma = \mathbb{P}^1$ which is the case we are interested in, the $\omega_{g,n}$ are proven to be symmetric meromorphic $n$-differentials on $\Sigma$ that are polynomials in the parameter $\bb$ of degree at most $2 g$ (see \cref{sec:RTR} for their definition). Upon setting $\bb = 0$, we recover the usual Chekhov--Eynard--Orantin topological recursion correlators.

\subsection*{Refined topological recursion for $\bb$-Hurwitz numbers}
Our first main result is that for certain rational weights $G(z)$, $G$-weighted $\bb$-Hurwitz numbers can be computed by  refined topological recursion.  In order to state the result, let $G(z)$ be one of the following rational weights:
\begin{equation}\label{eq:allowedweights}
	G(z) = \frac{(u_1+z)(u_2+z)}{v-z},  \quad (u_1+z)(u_2+z), \quad \frac{1}{v-z}, \quad \text{or} \quad \frac{u_1+z}{v-z}  .
\end{equation}
 We consider the corresponding refined spectral curve $\mathcal S_{\bm \mu}$ defined as follows. Set $ \Sigma  = \mathbb P^1 $ and choose the meromorphic functions $x(z) = t \frac{G(z)}{z} $ and $ y(z) = \frac{z}{x(z)} $, and the data $\mathcal P_+ $ as
\[
\mathcal P_+  = \left\{0,-u_1,-\frac{u_1u_2}{u_1+u_2+v}\right\}, \quad \left\{0,-u_1\right\}, \quad \left\{0\right\}, \quad \text{or} \quad \left\{0,-u_1\right\},
\] for $G$ as in \eqref{eq:allowedweights} respectively. In  each of the above cases, we choose $\mu_a = 0$ for all $a \in \mathcal P_+$ except 
$\mu_0 = -1  $ to complete the definition of the refined spectral curve.

\begin{theorem}\label{thm:intro1}
	For each of the weights $G(z)$ given in \eqref{eq:allowedweights}, the refined topological recursion correlators $\omega_{g,n}$ on the corresponding refined spectral curve $\mathcal S_{\bm \mu}$ are generating functions for  $G$-weighted $\bb$-Hurwitz numbers. More precisely, for any $(g,n) \in \frac{1}{2} \mathbb Z_{\geq 0} \times \mathbb Z_{\geq 1}$ except $(g,n) = (\frac{1}{2},1)$,
	expanding the $\omega_{g,n}(z_1,\ldots,z_n)$ near the point $z_i = 0$ in the local coordinate $x(z_i)^{-1}$ gives 
	\[
	\omega_{g,n}(z_1, \ldots, z_n) = \sum_{\mu_1,\dots,\mu_n \geq 1} F_{g,n}[\mu_1,\ldots,\mu_n] \prod_{i=1}^n \frac{dx(z_i)}{x(z_i)^{\mu_i+1}} ,
	\] where the $F_{g,n}$ are the expansion coefficients of the tau function $\tau^{(\bb)}_G$ in \eqref{tau in terms of Fgn}.   The above formula also holds for $\omega_{\frac{1}{2},1}$ up to certain explicit corrections (see \cref{thm:w/o} and \cref{thm:w/oother}).
\end{theorem}
Our theorem is the first known result that the generating function of $\bb$-Hurwitz numbers has an analytic continuation to a rational algebraic curve (the spectral curve). Simultaneously, restricted to degree-2 curves, we extend the topological recursion result of \cite{BychkovDuninBarkowskiKazarianShadrin2024, AlexandrovChapuyEynardHarnad2018} and the study of rational parametrizations in enumerative combinatorics \cite{Gao1993,Chapuy2009,GouldenGuay-PaquetNovak2013a} to new models with non-orientable topologies --- applications to higher degree curves are wide-open. It is worth noting that the structure of this analytic continuation differs substantially from the $\bb = 0$ case -- indeed, the $\omega_{g,n}$ have poles along the ``anti-diagonals'' $z_i = \sigma(z_j) $  for any $i,j \in [n]$ (and at a subset of $\mathcal{P}$ for the models with internal faces which will be discussed below). The weights $G(z) = (u_1+z)(u_2+z)$ and $G(z) = \frac{1}{v-z}$ covered in the above theorem correspond to $\bb$-bipartite maps and $\bb$-monotone Hurwitz numbers respectively~\cite{ChapuyDolega2022,BonzomChapuyDolega2023}.

\subsection*{Application to $\beta$-ensembles} As an application  of \cref{thm:intro1}, we prove in \cref{subsec:BetaEns} that refined topological recursion computes the correlators of three important models of $\beta$-ensembles. A $\beta$-ensemble is defined by a certain probability measure $d \mu^V_{N;\beta}$ on $\mathbb R^N$ which depends on a  function $V$ known as the potential.  We consider  three classical cases: 
\begin{itemize} 
	\item \textit{Gaussian $\beta$-ensemble} where $V(x) = \frac{x^2}{2}$;
	\item  \textit{Jacobi $\beta$-ensemble}  where $V(x) = \frac{1}{N}\left(\frac{2}{\beta}-c\right)\log(x)+\frac{1}{N}\left(\frac{2}{\beta}-d\right)\log(1-x)$ with $c,d>0$;
	\item and, \textit{Laguerre $\beta$-ensemble}   where $V(x) =
x+\frac{1}{N}\left(\frac{2}{\beta}-c\right)\log(x)$ with $c>0$;
\end{itemize} which correspond to the spectra of tridiagonal random matrices constructed in~\cite{DumitriuEdelman2002}.

In random matrix theory, one is interested in computing the connected correlators $\left\langle \sum_{k=1}^N\lambda_k^{k_1},\dots, \sum_{k=1}^N\lambda_k^{k_n}\right\rangle_{\mu_{N;\beta}^V}^\circ$ which we can encode in the following formal power series in $t$:
\begin{equation}\label{MMcorrelators}
	W^V_n(x_1,\dots,x_n) := \sum_{k_1,\dots,k_n \geq 1}\prod_{i=1}^n
	\frac{dx_i}{x_i^{k_i+1}}\left\langle \sum_{k=1}^N(t\lambda_k)^{k_1},\dots, \sum_{k=1}^N(t\lambda_k)^{k_n}\right\rangle_{\mu_{N;\beta}^V}^\circ.
\end{equation}
When the potential $V$ is chosen appropriately, various $\beta$-ensembles, including the three classical models above,  admit a  topological expansion  in $1/N$ as $N \to \infty$ \cite{BorotGuionnet2013}, whose expansion coefficients can be identified as weighted $\bb$-Hurwitz numbers \cite{Ruzza2023}. Thus, applying \cref{thm:intro1} (and  \cref{thm:RTRGbetaE} for  the Gaussian case), we prove  in \cref{thm:ensemble} that the topological expansion of the correlators $W^V_n$ can be computed by refined topological recursion. An abbreviated version of this statement reads as follows.
\begin{corollary}\label{cor:intro1}
	With the identification $\bb =   \sqrt{\frac{\beta}{2}}-\sqrt{\frac{2}{\beta}}$, the correlators $W^V_n$ of the Gaussian, Jacobi and Laguerre $\beta $-ensembles can be computed by refined topological recursion. More precisely,  for any $ n>1$, we have
	\[
		 \sqrt{\frac{\beta}{2}}^nW^V_n(x(z_1),\dots,x(z_n))=
		\sum_{g\in\frac12\mathbb{Z}_{\geq0}}\left(\sqrt{\frac{\beta}{2}}N\right)^{2-2g-n}\omega_{g,n}(z_1,\ldots,
		z_n) ,
	\] as a series expansion near $z_i = 0$, when $N\to \infty$. The  correlators  $\omega_{g,n}$  are computed by refined topological recursion on a refined spectral curve $\mathcal S_{\bm \mu}$ that depends on the $\beta$-ensemble considered (see  \cref{thm:ensemble} for the  definition of $\mathcal S_{\bm \mu}$ in each of the three cases). The above formula also holds for $n = 1$ up to an explicit correction for $g = \frac{1}{2}$. 
\end{corollary} It is worth noting that the  formula of \cref{cor:intro1} holds as an equality of formal power series in $t$ such that for any fixed power of $t$, the LHS is an analytic function of $N$ whose expansion as $N \to \infty$ is given by the RHS. In particular, we emphasize that we do not need to treat the $\beta$-ensembles formally in $1/N$. Our theorem makes the recursive structure found in \cite{ChekhovEynard2006a} (in the formal setting) mathematically precise, and gives an explicit description of the analytic structure of the correlators using the formalism of refined topological recursion.

\subsection*{Inserting internal faces}

Our second main result is an extension of \cref{thm:intro1} to the case of $\bb$-Hurwitz numbers with internal faces. The terminology ``internal faces" stems from the interpretation of  $\tau^{(\bb)}_{G}$ as a generating function for  colored monotone Hurwitz maps, for  rational weight $G$ \cite{BonzomChapuyDolega2023, Ruzza2023}. The expansion coefficient $F_{g,n}[\mu_1,\ldots,\mu_n]$  is (up to an  automorphism factor) a weighted count of colored monotone Hurwitz maps on non-oriented surfaces of genus $g$ with $n$ faces of degrees $\mu_1,\cdots, \mu_n$. In this interpretation, the dependence on $\bb$ is through the factor $(-\sqrt{\alpha} \bb)^{\rho(\mathcal M)}$, where $\rho(\mathcal M)$ is the so-called measure of non-orientability of a map $\mathcal M$  -- see \cref{sec:combint} for  precise  statements.

We are interested in counting colored monotone Hurwitz maps of genus $g$ with $n$ marked boundary faces of degrees $\mu_1,\cdots, \mu_n $, where we allow the maps to also have arbitrarily many internal faces of degree at most $D \in \mathbb Z_{>0}$. We denote the appropriately weighted generating series of such maps, where an internal face of degree $i$ is weighted by $\sqrt{\alpha} \epsilon \tilde p_i$, by $F^D_{g,n}[\mu_1,\ldots,\mu_n;\epsilon]$, so that the exponent of the variable $\epsilon$ counts the number of internal faces. The equivalent algebraic definition is given by the following formula:
\begin{equation}\label{eq:FDgnintro}
		F^D_{g,n}[\mu_1,\ldots,\mu_n;\epsilon] := [\hbar^{2g-2+n}] \left( \prod_{i=1}^n \mu_i \frac{\partial}{\partial \tilde{p}_{\mu_i}} \cdot \log \tau^{(\bb)}_G \right) \Bigg|_{\substack{\tilde p_{ i} \mapsto \frac{\epsilon}{\hbar}p_{ i} \,  \forall \,  i  \leq D,\\ \tilde p_{ i} = 0\,  \forall \,  i >D}}.
\end{equation}
 Note in particular that these 	$F^D_{g,n}$ are also functions of $p_1,\ldots, p_D$ and $\epsilon$. When $\bb = 0$, the problem of counting weighted Hurwitz numbers with internal faces was considered in \cite{BychkovDuninBarkowskiKazarianShadrin2025, BonzomChapuyCharbonnierGarcia-Failde2024}, and was proved to be governed by topological recursion.

We prove that when $G$ is one of the  weights in \eqref{eq:allowedweights}, $G$-weighted $\bb$-Hurwitz numbers  with internal faces are computed by  refined topological recursion. The definition of the corresponding genus zero refined spectral curve, denoted $\mathcal S^D_{\bm \mu} = (\mathbb P^1, X(z),Y(z), \mathcal P_+, \{\mu_a\}_{a\in  \mathcal{P}_+}))$, is rather technical and can be found in \cref{def:refinedSCint} and \cref{sec:other weights}.

\begin{theorem}\label{thm:intro2}
		For any of the weights $G(z)$ given in \eqref{eq:allowedweights}, the refined topological recursion correlators $\omega_{g,n}$ on the corresponding refined spectral curve $\mathcal S^D_{\bm \mu}$ are generating functions for  $G$-weighted $\bb$-Hurwitz numbers with internal faces of degree at most $D$. More precisely, for any $(g,n) \in \frac{1}{2} \mathbb Z_{\geq 0} \times \mathbb Z_{\geq 1}$ except $(g,n) = (0,1),\;(\frac{1}{2},1)$,
	expanding the $\omega_{g,n}(z_1,\ldots,z_n)$ near the point $z_i = 0$ in the local coordinate $X(z_i)^{-1}$ gives 
	\[
	\omega_{g,n}(z_1, \ldots, z_n) = \sum_{\mu_1,\dots,\mu_n \geq 1} F^D_{g,n}[\mu_1,\ldots,\mu_n;\epsilon] \prod_{i=1}^n \frac{dX(z_i)}{X(z_i)^{\mu_i+1}}
	\] where the $F^D_{g,n}$ are defined in \eqref{eq:FDgnintro}.   The above formula also holds for $\omega_{0,1}$ and $\omega_{\frac{1}{2},1}$ up to certain explicit corrections (see \cref{thm:w/} and \cref{thm:w/other}).
\end{theorem}

When the  weight  is $G(z) = (u_1+z)(u_2+z)$, \cite{ChapuyDolega2022} proves that the coefficients $F_{g,n}^D[\mu_1,\ldots,\mu_n;\epsilon]$ are the $\bb$-weighted generating series of bipartite maps on a non-oriented surface of genus $g$ with $n$ boundary faces of degrees $\mu_1,\cdots, \mu_n$ and arbitrarily many internal faces of degree at most $D$.  Another important case in which a  result analogous to \cref{thm:intro2} holds is the enumeration of (not necessarily bipartite) maps with internal faces, which we  address in \cref{sec:Appendix}. We  extend the previously known results on rational parametrization of generating series of maps and bipartite maps to arbitrary topologies and weighted internal faces of arbitrary bounded degrees. The precise statements for bipartite maps and maps can be found in  \cref{theo:GenSerMaps}.

\subsection*{Context and future work}
To prove  \cref{thm:intro1}, we use   Virasoro-type constraints found in \cite{ChidambaramDolegaOsuga2024} that uniquely determine the tau function $\tau^{(\bb)}_G$. We turn these constraints into formal loop equations for the generating functions  of the $ F_{g,n}$  and prove that the refined topological recursion correlators $\omega_{g,n} $ uniquely solve these loop equations using the results of \cite{KidwaiOsuga2023}. Our result extends the result of \cite{AlexandrovChapuyEynardHarnad2018, BychkovDuninBarkowskiKazarianShadrin2024} to arbitrary $\bb$ for the restricted choice of  weights $G(z)$ in equation \eqref{eq:allowedweights}. However, the results of \textit{loc. cit.} rely heavily on KP integrability, which is unknown in the $\bb$-deformed setting and hence we cannot use their methods to prove our result.

The strategy of our proof applies to any rational weight $G(z)$, and indeed can be used to show that  the generating function of rationally weighted $\bb$-Hurwitz numbers analytically continues to a rational curve $\Sigma = \mathbb P^1$ (extending the result of  \cite{BorotChidambaramUmer2025} to arbitrary $\bb$). The only obstruction is that the refined topological recursion formalism is yet to be defined when the associated branched covering $x : \Sigma \to \mathbb P^1$ is of degree greater than $2$. We will return  to this question in the future.

The key idea in extending \cref{thm:intro1} to \cref{thm:intro2} for $\bb$-Hurwitz numbers with internal faces is that inserting internal faces in the combinatorial models can be interpreted as the application of the variational formula in topological recursion. This technique was employed by \cite{BonzomChapuyCharbonnierGarcia-Failde2024} in the $\bb=0$ setting, and we use the refined variational formula proved in \cite{Osuga2024} to extend the result to arbitrary $\bb$. 

In his list of unsolved problems in map enumeration~\cite{Bender1991}, Bender posed a question regarding the origin of a universal pattern that has been discovered and proven for various models of maps over the years. This universal pattern matches physics predictions from the so-called double scaling limits in quantum gravity. Eynard explained in~\cite[Chapter 5]{Eynard2016} how the structural properties of the topological recursion correlators can be used to prove these predictions. In particular, the recent results in the $\bb=0$ setting on topological recursion applied to weighted Hurwitz numbers with internal faces provide a general answer to Bender's question in the case of orientable surfaces. However, in the non-orientable case, this asymptotic pattern is proved only for some special models, and remains conjectural in general. We plan to use refined topological recursion to tackle this problem in future work.

\subsection*{Acknowledgments}
We would like to thank Guillaume Chapuy  and Elba Garcia-Failde for many helpful discussions.

\section{$\bb$-Hurwitz numbers}\label{sec:b-Hurwitz}

\subsection{Jack polynomials}

Let $J^{(\alpha)}_\lambda$ be the \emph{integral version of the Jack symmetric
function} indexed by an integer partition $\lambda = (\lambda_1 \geq
\dotsm \geq \lambda_\ell)$ and a parameter $\alpha$. There are many equivalent definitions
of these symmetric functions, and the standard references
are~\cite{Stanley1989,Macdonald1995}. We will follow the presentation
from~\cite[Section 2.1]{ChidambaramDolegaOsuga2024} that is well suited for our purposes. After identifying the normalized power sum symmetric
functions
\[\sqrt{\alpha}^{-1}\sum_{j \geq 1}x_j^i\]
with the formal variable
$\tilde{p}_i$, we can characterize $J^{(\alpha)}_\lambda$ as
eigenfunctions of a certain differential operator acting on
the graded polynomial algebra $\Q(\sqrt{\alpha})[\tilde{\pp}] := \Q(\sqrt{\alpha})[\tilde{p}_1, \tilde{p}_2,\dotsc]$ with the
grading $\deg(\tilde{p}_i) := i$. Let
\begin{equation*}
  \label{eq:hrep}
    \begin{split}
      \mathsf J_{k} = \left\{
        \begin{array}{lr}
          \hbar k \partial_{\tilde p_k}& k \geq 0,\\
          0 &  k = 0,\\
          \hbar \tilde p_{-k} &  k <0. 
        \end{array}
        \right.
    \end{split}
  \end{equation*}
be a representation of the Heisenberg algebra acting on
$\Q(\sqrt{\alpha})[\tilde{\pp}]$. Then, the Jack
 symmetric functions $J^{(\alpha)}_\lambda$ are the unique homogeneous
 elements of $\Q(\sqrt{\alpha})[\tilde{\pp}]$ of degree
 $|\lambda|:=\lambda_1 +\dotsm+\lambda_\ell$ such that:
 \begin{itemize}
 \item[\rm \textbf{J1:}]
   $ J^{(\alpha)}_\lambda$ is an eigenfunction of the
     Laplace--Beltrami operator
     \begin{equation*}
  \label{eq:Laplace--BeltramiDef}
  \tilde{D}_\alpha := \frac{\hbar^{-3}}{2}\sum_{k\geq1}\J_{-k}\left(\sum_{\ell\geq1}\J_{\ell}  \J_{k-\ell}
   -(k-1)\hbar\cdot\bb\cdot\J_{k}\right), \quad \bb := \sqrt{\alpha}^{-1}-\sqrt{\alpha},
\end{equation*}
with  eigenvalue equal to
\[ \sum_{(x,y) \in \lambda}\tilde{c}_\alpha (x,y),\quad
  \text{where}\quad \tilde{c}_\alpha (x,y) := \sqrt{\alpha}(x-1)-\sqrt{\alpha}^{-1}(y-1);\]
  where $(x,y)\in\lambda$ (which we occasionally write as $\square\in\lambda$ for short) denotes the box in column $x$ and row $y$ of the Young diagram of $\lambda$.
\item[\rm \textbf{J2:}]
  the transition matrix from $\left\{J^{\al}_\lambda\right\}_{|\lambda|=n}$
    to the monomial symmetric functions $\left\{m_\lambda\right\}_{|\lambda|=n}$
    is lower triangular;
\item[\rm \textbf{J3:}]
    $J^{(\alpha)}_\lambda$ is normalized such that
$[\tilde{p}_1^{|\lambda|}]J^{(\alpha)}_\lambda$ is equal to $\sqrt{\alpha}^{|\lambda|}$; with this normalization, we denote $j_\lambda^{(\alpha)} := \langle J^{\al}_\lambda,J^{\al}_\lambda \rangle$, where $\langle \tilde{p}_\lambda,\tilde{p}_\mu \rangle := \delta_{\lambda,\mu}\frac{|\lambda|!}{\Aut(\lambda)}$.
\end{itemize}

\begin{remark}
  The use of $\tilde{p}_i$ instead of usual (non-normalized) power sum symmetric
functions $p_i = \sqrt{\alpha} \tilde p_i$, as well as using the parameter
$\bb$ instead of $b:=\alpha-1$ is not standard. However, we will see later that this
normalization is more natural in the context of refined
topological recursion.
\end{remark}

\subsection{$\bb$-Hurwitz numbers: combinatorial interpretation}\label{sec:combint}

Let $G(z) = \sum_{i=0}^\infty g_i z^i$ be an invertible formal power series. We
define the tau function of $G$-weighted $\bb$-deformed single Hurwitz
numbers as the following formal power series:
\begin{multline}
  \label{eq:TauFunction}
    \tau_G^{{(\bb)}}=\sum_{d \geq 0} \left(\frac{t \cdot \sqrt{\alpha}}{\hbar} \right)^d\sum_{\lambda \vdash d} 
		\frac{J_\lambda^{(\alpha)}(\sqrt{\alpha}\tilde{\pp})}{j_\lambda^{(\alpha)}}\prod_{\square
                  \in \lambda}G(\hbar\cdot
  \tilde{c}_{\alpha}(\square))= \\
  \exp\left(\sum_{g\in\frac12\mathbb{Z}_{\geq0}}\sum_{n\in\mathbb{Z}_{\geq1}}\sum_{\mu_1,\ldots,\mu_n\in\mathbb{Z}_{\geq1}}\frac{\hbar^{2g-2+n}}{n!}\;
  F^G_{g,n} [\mu_1,\ldots,\mu_n] \;\frac{\tilde p_{\mu_1}}{\mu_1}\cdots
  \frac{\tilde p_{\mu_n}}{\mu_n}\right).
\end{multline}

This function was introduced in~\cite{ChapuyDolega2022} in the more
general case of $\bb$-deformed double Hurwitz
numbers, where the topological/combinatorial interpretation of $F^G_{g,n}$ as the
generating series of weighted real (non-oriented) branched coverings of genus
$g$ with $n$ boundaries was
proved. In the special case $\bb=0$ it reduces to the tau function of the 2D
Toda hierarchy corresponding to the generating series of $G$-weighted
double Hurwitz numbers studied in~\cite{Guay-PaquetHarnad2017}. In the
case of rationally weighted single $\bb$-Hurwitz numbers the partition function
$\tau^{\bb}_{G}$ is the generating function of colored monotone Hurwitz numbers introduced
in~\cite{BonzomChapuyDolega2023} for the monotone case and extended
in~\cite{Ruzza2023} to the general case\footnote{In case of a
  polynomial weight, there is a different model of non-orientable
  constellations introduced in~\cite{ChapuyDolega2022} that we will
  discuss in \cref{sec:applications}}.

\begin{definition}[Colored monotone Hurwitz map]
	A \emph{monotone Hurwitz map} $\mathcal{M}$ of genus $g$ with $v(\M) = d$
        vertices and $e(\M) = r$
        edges is a 2-cell embedding of a loopless labeled multigraph
        on a compact surface (orientable or not) of genus $g$, with the following properties: 
\begin{itemize}
\item[\rm \textbf{HM1:}] the vertices of the map are labeled from $1$
  to $d$, and the neighborhood of each vertex is equipped with an
  orientation, which defines a local orientation at every vertex. Moreover each vertex has a distinguished corner called \textbf{active};
\item[\rm \textbf{HM2:}] the edges of the map are labeled from $e_1$
  to $e_r$. We let $\mathcal{M}_i$ be the submap of $\mathcal{M}$ induced by edges $e_1,e_2,\dots,e_i$.
\item[\rm \textbf{HM3:}] for each $i$ in $\{1,\dotsc,r\}$, let
  $a_i<b_i$ be the two vertices incident to the edge $e_i$. Then $b_1
  \leq \dotsm \leq b_r$;
  \item[\rm \textbf{HM4:}] in
  the map $\mathcal{M}_{i}$, the following conditions (with respect to the local
  orientations inherited from $\mathcal{M}$) must be met: 
\begin{itemize}
\item[\rm \textbf{(i):}] the active corner at $b_i$ immediately follows the edge $e_i$;
		\item[\rm \textbf{(ii):}] the active corner at $a_i$ is opposite
                  (with respect to $e_i$) to the active corner at $b_i$;
		\item[\rm \textbf{(iii):}] if the edge $e_i$ is
                  disconnecting in $\mathcal{M}_i$, then the local orientations at $a_i$ and $b_i$ are compatible in $\mathcal{M}_i$ (i.e. they can be jointly extended to a neighborhood of $e_i$).
	\end{itemize}
\end{itemize}
The \textbf{degree} of a face is its number of active corners. These
face degrees form a partition of $d$ called the \emph{degree profile}
of $\mathcal{M}$. Note that the genus $g$\footnote{We use the convention that relates genus uniformly to the Euler characteristic: in the orientable case $g$ is equal to the number of tori attached to the sphere (hence $g\in\mathbb{Z}_{\geq0}$), while in the non-orientable case $2g$ is equal to the number of crosscaps attached to the sphere (hence $g\in\frac12\mathbb{Z}_{\geq0}$).} of a monotone Hurwitz map $\mathcal{M}$ with degree profile $(k_1,\dots,k_n)$ is given by the Riemann--Hurwitz formula
\[ \chi(\mathcal{M})=d-r+n = 2-2g,\]
where $\chi(\mathcal{M})$ denotes the Euler characteristic of a monotone Hurwitz map $\mathcal{M}$. 
A \emph{colored monotone Hurwitz map} $(\mathcal M, c)$ is a monotone Hurwitz map $\mathcal M$
equipped with an $(M|N)$-coloring $c\colon \{1,\dotsc,r\}
\longrightarrow \{1,\dotsc,M+N\}$, where $M,N \geq 0$, such that
\begin{itemize}
\item[$\bullet$] if $1 \leq i < j \leq r$ and $c(i) = c(j) \leq N$ then
  $b_i < b_j$;
  \item[$\bullet$] if $1 \leq i < j \leq r$ and $b_i = b_j$ then
    $c(i) \leq c(j)$.
	\end{itemize}
      \end{definition}

\begin{example}
Consider the monotone Hurwitz map depicted in \cref{fig:MonotoneColoredMap}. In the middle we represent it as an embedding into the Klein bottle, i.e. the non-orientable surface of genus $1$. We realize the Klein bottle by identifying the opposite edges of a rectangle, so that the identification is consistent with the black arrows. The map has $d=3$ vertices and their labels and local orientations are depicted in blue. It has $r=5$ edges and the active corners are indicated by the small red arrows. The green and yellow regions corresponds to the faces so that the degree profile is $(2,1)$. The same monotone Hurwitz map is represented on the right as a ribbon graph, which is a small open neighborhood of the graph embedded in the surface. Note that the choice of which ribbons are ``twisted" in this graphical representation is not unique. There is a unique $(0|1)$-coloring given by $c(i) \equiv 1$. There are $4$ possible $(1|1)$-colorings: $c(1) = c(3) = c(4) = 1$, $c(2),c(5) \in \{1,2\}$. However, there are no $(1|0)$ colorings, because $b_1 = b_2 = 2$, and the condition $c(1) < c(2)$ is violated.
\end{example}

\begin{figure}
\centering
\includegraphics[width=\textwidth]{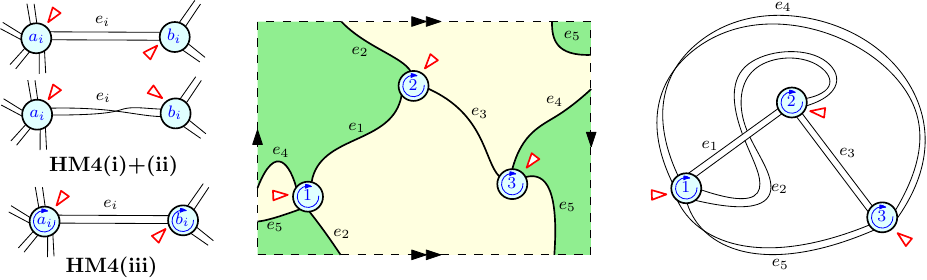}
\caption{On the left hand side we show a graphical representation of the conditions {\rm \textbf{HM4:(i)+(ii)}} and {\rm \textbf{HM4:(iii)}} for a monotone Hurwitz map represented as a ribbon graph. In the middle we represent a monotone Hurwitz map embedded into the Klein bottle. On the right hand side we show the same map as a ribbon graph.}
\label{fig:MonotoneColoredMap}
\end{figure} 

      In the following we  collect the colored monotone
      Hurwitz maps by the generating series that will additionally
      ``measure their non-orientability'' by the following algorithm.

      \begin{definition}
        \label{def:MON}
        Let $\mathcal{M}$ be a monotone Hurwitz map. We iteratively
        remove either the vertex of maximum label (if it is isolated) or
        the edge of maximal label. Doing so we also collect a
        weight $1$ or $-\sqrt{\alpha}\bb$ whenever we remove an edge $e$,
        by the following rules. Let $\mathcal{M}'$ be the map obtained
        from $\mathcal{M}$ by removing $e$. Then:
        \begin{itemize}
          \item[(i)] if $e$ joins two active corners in the same face of $\mathcal{M}'$, the weight is 1 if $e$ splits the face into
            two, and $-\sqrt{\alpha}\bb$ otherwise;
                      \item[(ii)] if $e$ joins two active corners
                        in different faces of $\mathcal{M}'$, the
                        weight is $1$ if the local orientations of the
                        active corners are compatible, and $-\sqrt{\alpha}\bb$ otherwise.
                      \end{itemize}
                      The product of all the collected weights is necessarily of the form
                      $(-\sqrt{\alpha}\bb)^{\rho(\mathcal{M})}$ for
                      some non-negative integer $\rho(\mathcal{M})$,
                      which is called the \textbf{measure of non-orientability}
                      of $\mathcal{M}$.
\end{definition}

\begin{example}
We continue the example from \cref{fig:MonotoneColoredMap}. By removing the edge $e_5$ the associated active corners belong to the same face of $\mathcal{M}_4$, and $\mathcal{M}_4$ has one face less than $\mathcal{M}_5$, so the associated weight is $1$. Passing from $\mathcal{M}_4$ to $\mathcal{M}_3$, the active corners belong to different faces of $\mathcal{M}_3$, so by comparing the local orientation of $1$ and $3$ joined by $e_4$, we see that it is not consistent, and the associated weight is $-\sqrt{\alpha}\bb$. By removing $e_3$ we split the map into two connected components, and in this case the weight is always $1$ as a consequence of {\rm \textbf{HM4:(iii)}}. Finally, by removing $e_2$ we are in case (i) which assigns the weight $1$, and at the end we end up with a single edge whose removal always gives the weight $1$ as it disconnects the map. The total weight of the map is $(-\sqrt{\alpha}\bb)$.
  \end{example}

\begin{theorem}{\cite{BonzomChapuyDolega2023,Ruzza2023}}
  \label{theo:Combi}
Let $G(z) =
\frac{\prod_{i=1}^N(1+u_iz)}{\prod_{i=N+1}^{M+N}(1-u_iz)}$. Then
\begin{equation}
    \frac{\alpha^g|\mu|!}{|\Aut(\mu)|\prod_{i=1}^n\mu_i}F_{g,n}[\mu_1,\ldots,\mu_n]=t^{|\mu|}\sum_{(\mathcal{M},c)}(-\sqrt{\alpha}\bb)^{\rho(\mathcal{M})}u_{c(1)}\dotsm u_{c(r)},\label{F-H}
  \end{equation}
  where we sum over all colored monotone Hurwitz maps of genus $g$ with the degree
  profile $(\mu_1,\dots,\mu_n)$.
\end{theorem}

\subsection{$\bb$-Hurwitz numbers with internal faces}
\label{subsec:InternalFaces}

In the following we would like to compute a variant of the  $\bb$-Hurwitz generating function considered in the previous section, where some of the faces will be treated as marked boundaries.

\begin{definition}
  \label{def:HurwitzWithInternal}
Let $\mathcal{M}$ be a monotone Hurwitz map with the degree profile
given by a permutation of the sequence
$(k_1,\dots,k_n,k_{n+1},\dots,k_{n+m})$. We say that
$\tilde{\mathcal{M}}$ is a monotone Hurwitz map with $n$ boundaries
and $m$ internal faces if $\tilde{\mathcal{M}}$ is obtained from
$\mathcal{M}$ by additionally marking $n$ active corners $c_1,\dots,c_n$ in
different faces of $\mathcal{M}$. We define the \emph{boundary degree
profile} $\deg_B(\tilde{\mathcal{M}}) := [k_1,\dots,k_n]$ where the corner
$c_i$ belongs to a face of degree $k_{i}$, and the \emph{internal degree
profile} 
$\deg_I(\tilde{\mathcal{M}}) := (k_{n+1},\dots,k_{n+m})$. The \textit{weight} of $\tilde{\mathcal{M}}$ is defined by assigning the weight 
$\epsilon t^i p_i$ to each internal face of degree $i$, i.e.,
\[ \wt(\tM) := \epsilon^m\prod_{i=n+1}^{n+m}t^{k_i} p_{k_i}.\]
\end{definition}

\begin{proposition}
  \label{prop:HurwitzWithBoundaries}
Let $D,n,k_1,\dots,k_n>0$. The following quantity is the weighted
generating function of colored monotone Hurwitz
maps with  boundary degree
$\deg_B(\tilde{\mathcal{M}}) := [k_1,\dots,k_n]$ and internal faces
of degree at most $D$:
\begin{align}
F_{g,n}^{D}\left[k_1,\ldots,k_n;\epsilon\right] :&= \sum_{m \geq
  0}\frac{\epsilon^m}{m!}\sum_{k_{n+1},\dots,k_{n+m}=1}^DF_{g,n}[k_1,\dots,k_n,k_{n+1},\dots,k_{n+m}]\prod_{i=n+1}^{m+n}\frac{p_{k_i}}{k_i}\nonumber\\
&= \alpha^{-g}t^{k_1+\dotsm+k_n}\sum_{(\tilde{\mathcal{M}},c)}\frac{\wt(\tM)}{v(\tilde{\mathcal{M}})!}\,(-\sqrt{\alpha}\bb)^{\rho(\tM)}\,u_{c(1)}\dotsm u_{c(r)}  \label{eq:HurwitzWithBoundaries}
\end{align}
where we sum over all colored monotone Hurwitz maps of genus $g$ with  boundary degree
  profile $[k_1,\dots,k_n]$ and internal faces of degree at most $D$.
\end{proposition}

\begin{proof}
\cref{theo:Combi} implies that the LHS of
\eqref{eq:HurwitzWithBoundaries} is given by 
\[
 \alpha^{-g}t^{k_1+\dotsm+k_n}\sum_{m \geq
  0}\sum_{k_{n+1},\dots,k_{n+m}=1}^D\sum_{(\mathcal{M},c)}\frac{k_1\dotsm
    k_n |\Aut(k_1,\dots,k_{n+m})|}{v(\mathcal{M})!m!}\,
  (-\sqrt{\alpha}\bb)^{\rho(\mathcal{M})}\prod_{i=n+1}^{m+n} (\epsilon t^{k _{i}} p_{k _{i}})\,u_{c(1)}\dotsm u_{c(r)},
\]
where we sum over colored monotone Hurwitz maps with 
degree profile $(k_1,\dots,k_{n+m})$. Fix such a map $\mathcal{M}$ and note
that one
can mark its $n$ corners to obtain a monotone Hurwitz map $\tilde{\mathcal
  M}$ with $n$ boundaries and  boundary degree profile
$[k_1,\dots,k_n]$ in $k_1\dotsm k_n \cdot
\frac{|\Aut(k_1,\dots,k_{n+m})|}{|\Aut(k_{n+1},\dots,k_{n+m})|}$ ways.
Therefore we can rewrite the above expression as a sum over all colored monotone Hurwitz maps with $n$
boundaries,  boundary degree
  profile $[k_1,\dots,k_n]$, and  internal degree
  profile $(k_{n+1},\dots,k_{n+m})$:
  \[\alpha^{-g}t^{k_1+\dotsm+k_n}\sum_{k_{n+1},\dots,k_{n+m}=1}^D\sum_{(\tM,c)}\frac{\wt(\tM)}{v(\tM)!}\frac{|\Aut(k_{n+1},\dots,k_{n+m})|}{m!}\,
  (-\sqrt{\alpha}\bb)^{\rho(\mathcal{M})}\,u_{c(1)}\dotsm u_{c(r)}.\]
We conclude the proof by noting that the orbit of the action of the permutation
group on the sequence
$[k_{n+1},\dots,k_{n+m}]$ has size
$\frac{m!}{|\Aut(k_{n+1},\dots,k_{n+m})|}$. Thus the expression above
is equal to the RHS of~\eqref{eq:HurwitzWithBoundaries}.
\end{proof}

\begin{remark}
  Note that for the choice of the weight $G(z) =
\frac{\prod_{i=1}^N(u_i+z)}{\prod_{i=N+1}^{M+N}(u_i-z)}$ we also have
combinatorial interpretations analogous to \cref{theo:Combi} and
\cref{prop:HurwitzWithBoundaries}. Indeed, in
order to get these interpretations it is enough to change variables
$t \mapsto t\frac{u_1\dotsm u_N}{u_{N+1}\dotsm u_{N+M}}$, $u_i \mapsto
  u_i^{-1}$ in \eqref{F-H} and
  \eqref{eq:HurwitzWithBoundaries}. In the following, we will use this
   parametrization as it produces less complicated equations, 
  is consistent with the parametrization used
  in~\cite{ChidambaramDolegaOsuga2024}, and is more suited to the
  applications discussed in~\cref{sec:applications}.
  \end{remark}

\subsection{Constraints for $\tau^{(\bb)}_{G}$}

We finish this section by stating one of the main results
of~\cite{ChidambaramDolegaOsuga2024}, which gives 
explicit constraints satisfied by $\tau^{(\bb)}_G$ in the case of 
rational weight $G$ as in \cref{theo:Combi}. These constraints are
related to representations of finite rank $\mathcal{W}$-algebras, and
they uniquely determine $\tau^{(\bb)}_G$. When $G(z)=\frac{(u_1+z)(u_2+z)}{v-z}$, the constraints read as follows.
 \begin{theorem}{\cite[Theorem 4.8]{ChidambaramDolegaOsuga2024}}
      \label{theo:CDO}
      For any $k \geq 0$ one has
    \begin{equation}
         D_k  \tau_G^{{(\bb)}} = 0\label{Dconstraints}
    \end{equation}
      where
      \begin{equation}
  \label{Voperator}
  D_k :=   t\left( \sum_{j\geq 1}  \mathsf J_{k-j}  \mathsf J_j
  \right)+ t(u_1+u_2-\hbar \bb k) \mathsf J_k + t u_1 u_2 \delta_{k,0} +  \sum_{j \geq 0} \mathsf J_{k-j} \mathsf J_{j+1} -(v+\hbar \bb k) \mathsf J_{k+1}.
  \end{equation}
\end{theorem}
\noindent Henceforth, we  always work with the weight $G(z)=\frac{(u_1+z)(u_2+z)}{v-z}$ unless specified otherwise.

\section{Refined topological recursion}\label{sec:RTR}
We give a brief review of the refined topological recursion formalism as introduced in \cite{KidwaiOsuga2023,Osuga2024a} building on work done in \cite{ChekhovEynard2006a}. 

\subsection{Refined spectral curve}

For our purposes, we restrict to the setting of genus $0$ curves. We refer the readers to \cite[Section 2]{Osuga2024a} for higher-genus curves.

\begin{definition}[\cite{KidwaiOsuga2023,Osuga2024a}]\label{def:RSC}
A (genus-zero) \emph{refined spectral curve} $\mathcal{S}_{{\bm \mu}}$ consists of the following data:
\begin{itemize}
    \item[{\rm \textbf{S1:}}]  $(\Sigma,x,y)$: the Riemann sphere $\Sigma=\mathbb{P}^1$ with two non-constant meromorphic functions $(x,y)$ satisfying
    \begin{equation*}
        P(x,y) = 0
    \end{equation*}
    where $P(x,y)$ is an irreducible polynomial of degree $2$ in  $y$, such that $dx$ and $dy$ do not have common zeroes\footnote{This condition was forgotten to mention in \cite[Section 2]{Osuga2024a} but implicitly assumed.}. We use $\sigma:\Sigma\to\Sigma$ to denote the canonical global involution of $x:\Sigma\to\mathbb{P}^1$ and  $\mathcal{R} \subset \Sigma$ to denote the set of ramification points of $x$, i.e., the set of $\sigma$-fixed points.
    \item[{\rm \textbf{S2:}}]  $(\mathcal{P}_+,\left\{\mu_a\right\}_{a\in  \mathcal{P}_+})$: a choice of a decomposition $\mathcal{P}_+\sqcup\sigma(\mathcal{P}_+)=\mathcal{P}$ where $\mathcal{P}$ is the set of  zeroes and poles of $(y(z)-y(\sigma(z)))dx(z)$ which are not ramification points, and associated parameters $\mu_a\in\mathbb{C}$ for all $a\in\mathcal{P}_+$.
\end{itemize}
\end{definition}

Let us denote by $B$ the \emph{fundamental bidifferential} which is the unique symmetric bidifferential on $\Sigma$ with a double pole on the diagonal with biresidue $1$ and no other poles.
In any global coordinate $z$\footnote{By abuse of notation, we use $z$ both as a point in $\Sigma$ and  as a global coordinate on $\Sigma$. In our genus-zero setting this does not cause any issue.}, the bidifferential $B$ admits the following rational expression:
\begin{equation*}
    B(z_1,z_2)=\frac{dz_1dz_2}{(z_1-z_2)^2},
\end{equation*}
and it satisfies
\begin{equation}\label{eq:BsigmaB}
    B(z_1,z_2)+B(z_1,\sigma(z_2)) = B(z_1,z_2)+B(\sigma(z_1),z_2) = \frac{dx(z_1)dx(z_2)}{(x(z_1)-x(z_2))^2}.
\end{equation}
Furthermore, for $c\in\Sigma\setminus \mathcal{R}$, we define $\eta^c$ by
\begin{equation}\label{eq:etadef}
    \eta^c(z):=\int_{\sigma(c)}^cB(z,\cdot)
\end{equation}
which is a differential in $z$ with residue $\pm1$ at $z=c,\sigma(c)$ and no other poles.

\subsection{Refined topological recursion}

Before giving the definition of refined topological recursion, or RTR for short, let us  introduce some notation. 

\begin{definition}
Given a meromorphic function $x:\Sigma\to\mathbb{P}^1$ of degree two and a sequence of symmetric multidifferentials $\{\omega_{g,n}\}_{(g,n)  \in \frac12\mathbb{Z}_{\geq0}\times \mathbb{Z}_{\geq1}}$ on $\Sigma$ with coordinate $z$, we define for any $(g,n) \in \frac12\mathbb{Z}_{\geq0}\times \mathbb{Z}_{\geq0}$
\begin{multline}\label{eq:Recdef}
	Q^\omega_{g,1+n}(z;z_{[n]}):= \sum_{\substack{g_1+g_2=g \\ J_1\sqcup J_2=[n]}} \omega_{g_1,1+|J_1|}(z,z_{J_1})\,\omega_{g_2,1+|J_2|}(z,z_{J_2}) +\sum_{i=1}^n\frac{\omega_{g,n}(z,z_{[n]}\backslash\{z_i\}) dx(z) dx(z_i)}{(x(z)-x(z_i))^2}\\
	+\omega_{g-1,2+n}(z,z,z_{[n]})  +\bb\,dx(z)\, d_{z}\frac{\omega_{g-\frac12,n+1}(z,z_{[n]})}{dx(z)},
\end{multline}
where $(z,z_{[n]}) \in \Sigma^{n+1}$, $d_z$ is the exterior derivative with respect to $z$, and whenever $\omega_{0,1}$ appears, we replace it by the antisymmetrized version $\frac{1}{2}\left(\omega_{0,1}(z) - \omega_{0,1}(\sigma(z))\right)$ with respect to $\sigma$, the canonical involution of $x$. We similarly define $\text{Rec}^\omega_{g,1+n}$ by excluding the terms with $(g_1,J_1)=(g,[n])$ or $(g_2,J_2)=(g,[n])$.
\end{definition}

\begin{remark}\label{rem:formal Q}
In Section \ref{sec:w/o}, we will use the same notation $Q^\phi$ and $\operatorname{Rec}^{\phi}$ for $\phi_{g,n}$ which are germs of meromorphic functions on a formal disk with coordinate $x^{-1}$, where we choose the required meromorphic function to be $x$. In this formal case, the involution operator $\sigma$ is not defined, and we keep $\phi_{0,1}$ in $Q^\phi$ or $\operatorname{Rec}^{\phi}$ without replacing.
\end{remark}

The objects $Q_{g,1+n}(z;z_{[n]})$ and $\text{Rec}_{g,1+n}(z;z_{[n]})$ are  quadratic differentials in the first entry $z \in \Sigma$, and differentials in the other entries. For readers familiar with the usual Chekhov--Eynard--Orantin topological recursion, it is worth noting that the second sum in the first line is merely an artifact of our unconventional definition of $\omega_{0,2}$ (see equation \eqref{def02} below). The main difference to note is the $\bb$-dependent term appearing in \eqref{eq:Recdef}.

We are ready to give the definition of refined topological recursion. Given a refined spectral curve $\mathcal S_{\bm \mu}$,  refined topological recursion will produce a sequence of multidifferentials $\omega_{g,n}$, where $(g,n) \in \frac{1}{2} \mathbb Z_{\geq 0} \times \mathbb Z_{\geq 1}$, on the Riemann surface $\Sigma$ underlying the spectral curve.  First, we  define the unstable correlators $\omega_{g,n}$ where  $(2g-2+n) \leq 0$.
\begin{definition}[\cite{KidwaiOsuga2023,Osuga2024a}]
Given a refined spectral curve $\mathcal S_{\bm \mu}$, we define the  differentials $\omega_{0,1}(z), \omega_{\frac12,1}(z) $ and the bi-differential $\omega_{0,2}(z_1,z_2)$ on the underlying Riemann surface $\Sigma$, which we will often refer to as the  \textit{unstable refined topological recursion correlators}:
\begin{align} 
	\omega_{0,1}(z):=&y(z)dx(z), \quad
	\omega_{0,2}(z_1,z_2):=-B(z_1,\sigma(z_2)),\label{def02}\\ \label{eq:1/2def}
	\omega_{\frac12,1}(z):=&\frac{\bb}{2}\left(-\frac{d\left(y(z)-y(\sigma(z))\right)}{y(z)-y(\sigma(z))}+\sum_{c\in\mathcal{P}_+}\mu_c\,\eta^c(z)\right). 
\end{align}
\end{definition}
\noindent Given these unstable correlators, the refined topological recursion is a formula that constructs the stable correlators $\omega_{g,n}$ for $(2g-2+n) >0$.

\begin{definition}[\cite{KidwaiOsuga2023,Osuga2024a}]
Given a genus zero refined spectral curve $\mathcal{S}_{\bm \mu}$, and the associated unstable correlators,  we define the  \emph{stable refined topological recursion correlators} $\omega_{g,1+n}(z_0,z_{[n]})$ for $(g,n) \in \frac{1}{2}\mathbb Z_{\geq 0} \times \mathbb Z_{\geq 0}$ such that $(2g-2+n) \geq 0$ as the following multidifferentials:
    \begin{equation}\label{eq:RTR}
        \omega_{g,1+n}(z_0,z_{[n]})=\frac{1}{2\pi \rm i} \oint_{z\in C_+-C_-}\frac{\eta^z(z_0)}{2(\omega_{0,1}(z)-\omega_{0,1}(\sigma(z)))}{\rm Rec}^\omega_{g,1+n}(z;z_{[n]}),
    \end{equation}
where  $C_+$ is a contour that encircles all the points of $\mathcal{P}_+\cup z_0 \cup z_{[n]} $ inside but none of the points of $\mathcal{R}\cup\sigma(\mathcal{P}_+)\cup\sigma(z_0) \cup \sigma(z_{[n]})$ and $C_-$ encircles all the points of $\mathcal{R}\cup\sigma(\mathcal{P}_+)\cup\sigma(z_{[n]})$ inside but none of the points of $\mathcal{P}_+\cup z_{[n]}$. This formula is recursive on the integer $ (2g-2+n)$, and is called the \emph{refined topological recursion}.
\end{definition}

The collection of multidifferentials $\omega_{g,n}$ for $(g,n) \in \frac{1}{2} \mathbb Z_{\geq 0} \times \mathbb Z_{\geq 1}$ (i.e., unstable and stable  correlators together) are called the \textit{refined topological recursion correlators}, or RTR correlators for short. Notice that the variable $z_0$ is treated non-symmetrically in the refined topological recursion formula and hence it is not clear, a priori, that  the $\omega_{g,n}$ constructed by the refined topological recursion are in fact symmetric multidifferentials. Along with this symmetry, a few other important properties regarding the structure of the poles of the correlators are proved in \cite{KidwaiOsuga2023,Osuga2024a}.

\begin{theorem}[Pole structure of the $\omega_{g,n}$] The RTR correlators $\{\omega_{g,n}\}_{g,n}$ satisfy the following properties:
     \begin{itemize}[leftmargin=16mm]
        \item The correlators $\omega_{g,n}$ are symmetric meromorphic multidifferentials.
        \item The  stable correlators $\omega_{g,n} (z_{[n]})$ have possible poles in any variable, say $z_i$, at the points $\mathcal{R} \cup \sigma(\mathcal{P}_+) \cup \sigma(z_{[n]})$ except the poles of the anti-invariant part of $ydx$ w.r.t. $\sigma$.
        \item The  stable correlators $\omega_{g,n}$ are residue free, i.e., at any pole, they do not have residues.
        \item The possible poles of $\omega_{\frac{1}{2},1}$ are at the points $ \mathcal R \cup \mathcal P$. 
    \end{itemize}
\end{theorem}  

The RTR correlators depend on a parameter $\bb \in \mathbb C$, and hence it's worth noting the following properties concerning the dependence on $ \bb$ \cite{KidwaiOsuga2023,Osuga2024a}.
\begin{proposition}[$\bb$-dependence of the $\omega_{g,n}$]  As functions of $\bb$, the RTR correlators have the following properties:
	\begin{itemize}
		\item The correlators $\omega_{g,n}$ are polynomials in $\bb$ of degree at most $2g$.
		\item The refined topological recursion correlators reduce to  the Chekhov--Eynard--Orantin topological recursion correlators \cite{ChekhovEynardOrantin2006,EynardOrantin2007} under the substitution $\bb=0$.
	\end{itemize}
\end{proposition}

\begin{remark}\label{rem:RTRequiv}
	It is  worth noting that one can express the refined topological recursion formula \eqref{eq:RTR} in a slightly different form:
	\begin{equation}\label{eq:RTRequiv}
		\omega_{g,1+n}(z_0,z_{[n]})=\frac{1}{2\pi \rm i} \oint_{z\in C_+}\frac{\eta^z(z_0)}{(\omega_{0,1}(z)-\omega_{0,1}(\sigma(z)))}{\rm Rec}^\omega_{g,1+n}(z;z_{[n]}),
	\end{equation} where we notice that the integrand as a function of $z$ only has poles at the points contained in the region bounded by $C_+ \cup C_-$. As the sum of the residues at all the poles of a meromorphic differential is zero, the contributions from the contour $-C_-$ equal the contributions from the contour $C_+$, which gives the above formula. 
\end{remark}

\subsection{Variational formula}

We often consider a family of refined spectral curves $\mathcal{S}_{\bm \mu}({\bm t})$ depending on a tuple of parameters ${\bm t}=(t_1,\ldots,t_{|{\bm t}|})$. In our genus-zero setting, this means that we restrict to the case when the underlying Riemann surface $\Sigma$ is always of genus $ 0$. The two meromorphic functions $x,y$ depend on ${\bm t}$ and consequently, all the correlators $\omega_{g,n}$ as well. Then, a natural question is to ask whether the $\{\omega_{g,n}\}_{g,n}$ behaves well in the family (i.e., under the variation of the parameters $\bm t$). In general, this is quite a difficult global question (see \cite{BorotBouchardChidambaramKramerShadrin2025} about when one can take limits in topological recursion in the $\bb = 0$ setting).

However, it turns out that $\{\omega_{g,n}\}_{g,n}$ behaves well (locally) under the so-called \emph{variation for fixed $x$}. Let us explain how this works. Choose $t\in{\bm t}$. For a meromorphic multidifferential $\omega$ on $\Sigma^n$, one can associate the corresponding function $\mathcal{W}$ by
\begin{equation}
    \omega(z_1,\ldots,z_n;t)=:\mathcal{W}(z_1,\ldots,z_n;t)\;dx(z_1;t)\cdots dx(z_n;t),
\end{equation}
where $d_i$ denotes the exterior derivative with respect to $z_i$ and $dx(z_i;t)=d_i x(z_i;t)$. 
Instead of naively considering the partial derivative with respect to $t$, we define the \emph{variational operator} $\delta_t^{(n)}$ with respect to $t$ by (c.f. \cite[Definition 3.1]{Osuga2024})
\begin{equation*}
    \delta_t^{(n)}\;\mathcal{W}(z_1,\dots,z_n;t):=\frac{\partial\mathcal{W}(z_1,\ldots,z_n;t)}{\partial t}-\sum_{i=1}^n\frac{\partial x(z_i;t)}{\partial t}\frac{d_i\mathcal{W}(z_1,\ldots,z_n;t)}{dx(z_i;t)}.
\end{equation*}
Similarly for a multidifferential, we define
\begin{equation*}
     \delta_t^{(n)}\;\omega(z_1,\dots,z_n;t):=dx(z_1;t)\cdots dx(z_n;t)\;\delta_t^{(n)}\;\mathcal{W}(z_1,\dots,z_n;t).
\end{equation*}
See  \cref{lem:alpha-derivative} to understand how the variational operator $\delta_t^{(n)}$ naturally arises in our context.

Henceforth, we drop the $t$-dependence in $x$ for brevity. With this definition, it is easy to check that
\begin{equation*}
    \delta_t^{(1)}x(p_1)=0,\qquad \delta_t^{(1)}dx(p_1)=0,
\end{equation*}
and furthermore, for a meromorphic function $f$ pulled back from the base curve $\mathbb{P}^1$, we have
\begin{equation*}
    \delta_t^{(n)}\;f(x(z_1),\dots,x(z_n);t)=\partial_t f(x(z_1),\dots,x(z_n);t).
\end{equation*}
Thus, $\delta_t^{(n)}$ indeed acts as the ``derivative for fixed $x$'' as in \cite{EynardOrantin2007}.

Let us denote by $\mathcal{P}^\infty$ the set of poles of $(y(z)-y(\sigma(z)))dx(z)$ which may be ramification points, and by $D_i+1$ the  order of the pole at $a_i\in\mathcal{P}^\infty$. Typically, ``good parameters'' of topological recursion with respect to which one would like to understand the variation of the correlators are associated with $\mathcal{P}^\infty$. More concretely, it is proven in \cite{Osuga2024} (and in \cite{EynardOrantin2007} when $\bb=0$) that the action of $\delta_t^{(n)}$ on $\omega_{g,n}$ is related to a certain integral of  $\omega_{g,n+1}$:

\begin{definition}[\cite{Osuga2024}]\label{def:deformation}
     We say that a refined spectral curve $\mathcal{S}_{\bm \mu}(\bm t)$ satisfies the \emph{refined deformation condition with respect to $t\in\bm t$} if there exists a contour $\gamma\subset\Sigma$ and a function $\Lambda$ holomorphic along $\gamma$ such that
    \begin{itemize}
        \item[\textbf{D1:}] The poles of $\Lambda$ lie in $\mathcal{P}^\infty$ and the  order of the pole at $a_i\in\mathcal{P}^\infty$ is $D_i$ or less.
        \item[\textbf{D2:}] The pair $(\gamma,\Lambda)$ satisfies
         \begin{align}
         \delta_t^{(1)}\left(\omega_{0,1}(z_1)-\omega_{0,1}(\sigma(z_1))\right)=&\int_{z\in\gamma}\Lambda(z)\,\left(\omega_{0,2}(z,z_1)-\omega_{0,2}(z,\sigma(z_1))\right),\label{BCCG1}\\
         \delta_t^{(2)}\omega_{0,2}(z_1,z_2)=&\int_{z\in\gamma}\Lambda(z)\,\omega_{0,3}(z,z_1,z_2)\\
         \delta_t^{(1)}\omega_{\frac12,1}(z_1)=&\int_{z\in\gamma}\Lambda(z)\,\omega_{\frac12,2}(z,z_1).
    \end{align}
    \end{itemize}
\end{definition}
\begin{theorem}[\cite{Osuga2024}]\label{thm:variation}
    If a refined spectral curve $\mathcal{S}_{\bm \mu}(\bm t)$ satisfies the refined deformation condition with respect to $t\in\bm t$, then for all $(g,n) \in \frac{1}{2}\mathbb Z_{\geq 0} \times \mathbb Z_{\geq 1}$,
    \begin{equation}\label{eq:varformula}
        \delta_t^{(n)}\omega_{g,n}(z_1,\ldots,z_n)=\int_{z \in\gamma} \Lambda(z) \,\omega_{g,1+n}(z,z_1,\ldots,z_n).
    \end{equation}
\end{theorem}\noindent If formula \eqref{eq:varformula} holds for some $t$,  we say that the $\omega_{g,n}$ \textit{satisfy the variational formula with respect to $t$}.

Let us mention that $(\gamma,\Lambda)$ in \cite{Osuga2024} is chosen to be anti-invariant under the involution $\sigma$, whereas we did not impose this condition here. There is no difference between the two in practice. This is because for $2g-2+n\geq1$, $\omega_{g,n+1}(z_0,z_{[n]})+\omega_{g,n+1}(\sigma(z_0),z_{[n]})$ has a zero of order $D_i$ at every $z_0=a_i\in\mathcal{P}$ as shown in \cite[Lemma 3.1]{Osuga2024a}, and thus only the anti-invariant parts contribute to the variational formula. We also  note that not all refined spectral curves satisfy the refined deformation condition, and the parameters $\mu_a$ need to be chosen appropriately. See \cite{Osuga2024} for some concrete examples.

\section{$\bb$-Hurwitz numbers from refined topological recursion}\label{sec:w/o}

In this section, we prove that refined topological recursion can be used to compute $G$-weighted $\bb$-Hurwitz numbers with the  weight 
\begin{equation}\label{weight}
    G(z)=\frac{(u_1+z)(u_2+z)}{(v-z)},
\end{equation} which, as before, is fixed throughout this section. Our approach is to derive a set of  loop equations (using \cref{theo:CDO}) satisfied by certain formal power series, called $\phi_{g,n}$, which are essentially generating series of $G$-weighted $\bb$-Hurwitz numbers (when $\bb=0$ similar approaches are taken in e.g. \cite{KazarianZograf2015}). Then, we will show that these generating functions analytically continue to meromorphic differentials on a certain rational curve. Finally, we show that these analytic continuations coincide with the RTR correlators $\omega_{g,n}$ on a certain refined spectral curve. 

\subsection{Formal loop equations}
We begin by rewriting the constraints \eqref{Dconstraints} on $\tau_G^{{(\bb)}}$ in terms of $F_{g,n}$ (which are the
$\bb$-Hurwitz numbers up to certain rescalings as given in
\eqref{F-H}). Recall that the   $F_{g,n}$ are the expansion coefficients of $\tau_G^{(\bb)}$ as in \eqref{eq:TauFunction}. Define the operator $\nabla(x)$ (sometimes referred to as the \textit{loop insertion operator}) as 
\begin{align*}
	\label{loop insertion}
	\nabla(x) := \sum_{k \geq 1} \frac{k\, dx}{x^{k+1}}\frac{\partial}{\partial \tilde p_k}.
\end{align*}

\begin{definition}
We define the
generating function of $F_{g,n}$ as the following formal series (which is a germ of meromorphic $n$-differentials on the $n$-th product of formal disks centered around $x = \infty$):
\begin{equation*}
    \phi_{g,n}(x_1,\dots,x_n):=\sum_{k_1,\ldots,k_n\geq1}F_{g,n}[k_1,\ldots,k_n] \prod_{i=1}^n \frac{dx_i}{x_i^{k_i+1}} =[\hbar^{2g-2+n}]\left(\prod_{i=1}^n\nabla(x_i)\cdot\log\tau_G^{(\bb)}\right)\Bigg|_{\mathbf{\tilde p} = 0},
 \end{equation*} for any $(g,n) \in \frac{1}{2}\mathbb Z_{\geq 0} \times \mathbb Z_{\geq 1}$.
\end{definition}
For the proofs in this section, it is convenient to add some shifts when $(g,n) = (0,1),(0,\frac{1}{2})$ as follows. 
Define the  shifted generating functions $\phi_{g,n}$ for any $(g,n) \in \frac{1}{2}\mathbb Z_{\geq 0} \times \mathbb Z_{\geq 1}$ as		
\begin{multline}\label{eq:phishifts}
		\tilde 	\phi_{g,n}(x_1,\ldots,x_n) := \phi_{g,n}(x_1,\ldots,x_n) + \delta_{(g,n),(0,1)} \left( \frac{t(u_1+u_2)-v x_1}{2x_1(t+x_1)} dx_1 \right) \\ 
			+   \delta_{(g,n),(\frac{1}{2},1)} \left(\frac{\bb \, dx_1}{2(t+x_1)}+\frac{\bb \,dx_1}{2x_1}\right).
		\end{multline}
Then, we have the following loop equations.

\begin{proposition}[Formal  loop equations]\label{prop:FLE}
For  $(g,n)\in \frac12\mathbb{Z}_{\geq0} \times \mathbb{Z}_{\geq0}$, the  generating functions $\tilde \phi_{g,1+n}(x,x_1,\ldots,x_n)$ satisfy
\begin{align} \label{formalRLE}
 Q^{{\tilde \phi}}_{g,1+n}(x;x_{[n]})= U_{g,1+n}(x,x_{[n]}) + (dx)^2 \sum_{i=1}^n d_{x_i} \left(\frac{x_i(t+x_i)}{x(t+x)}\frac{\tilde \phi_{g,n}(x_{[n]})}{(x-x_i) dx_i}\right),
\end{align} as formal power series in $x^{-1},x_1^{-1},\ldots, x_n^{-1}$ where we assume that $|x| < |x_i|$, $Q^{{\tilde \phi}}_{g,1+n}$ is defined in Remark \ref{rem:formal Q}, and we define
\begin{align}
    U_{0,1}(x):=&\frac{v^2x^2-2t(u_1v+u_2v+2u_1u_2)x+t^2(u_1-u_2)^2}{4x^2(t+x)^2} (dx)^2,\label{U_{0,1}}\\
    U_{\frac12,1}(x):=& -\frac{\bb \, v\,(dx)^2}{2x(t+x)} , \qquad
    U_{1,1}(x):=-\frac{t^2\bb^2\, (dx)^2}{4x^2(t+x)^2},\nonumber
\end{align}
and $U_{g,1+n}(x,x_{[n]}):=0$ for all other $(g,n)$. 
\end{proposition}
\begin{proof}
The statement is a consequence of the constraints of \cref{theo:CDO} that characterize the generating function $ \tau_G^{(\bb)}$. More precisely, equation \eqref{formalRLE} is a rewriting of 
\begin{equation}\label{eq:LE1}
    [\hbar^{2g+n}]  \left(\prod_{i=1}^n\nabla(x_i)\left(\frac{1}{\tau_G^{(\bb)}}\sum_{k\geq0}\frac{D_k (dx)^2 }{x^{k+2}} \tau_G^{(\bb)} \right)\right)\Bigg|_{\mathbf{ \tilde p}=0}=0.
\end{equation} We will give a  brief proof as the derivation of loop equations from differential constraints arising from $\mathcal W$-algebra representations has been worked out in various recent papers (see for instance \cite[Lemma 5.4]{BorotBouchardChidambaramCreutzig2024}, \cite[Section 4]{BorotKramerSchuler2024} or \cite[Section 2]{BorotBouchardChidambaramCreutzigNoshchenko2024}). Note that the operator $D_k$ \eqref{Voperator} is a sum of terms that are either linear or quadratic in the $\mathsf J_i$. Computing the contribution from the linear terms is  straightforward. For instance, we get
\[
	[\hbar^{2g+n}]  \left(\prod_{i=1}^n\nabla(x_i) \left(\frac{1}{\tau_G^{(\bb)}} \sum_{k > 0} \frac{\J_k dx}{x^{k+1}} \tau_G^{{(\bb)}}\right)\right)\Bigg|_{\mathbf{ \tilde p}=0}=  \phi_{g,1+n} (x,x_{[n]})
\] 
Let us compute the contributions coming from  the quadratic term $\sum_{j\geq 1} \J_{k-j} \J_j$  which appears in the operator $D_k$. First, we rewrite this term as
\begin{multline}\label{eq:JJterms}
	 \frac{1}{\tau^{(\bb)}_G} \sum_{k\geq 0} \sum_{j\geq 1} \frac{\J_{k-j}}{x^{k-j+1}}  \frac{\J_j}{x^{j+1}} \tau^{(\bb)}_G  \\ 
	 =  
	 \frac{1}{\tau^{(\bb)}_G} \left( \sum_{k ,j > 0} \frac{\J_{k}}{x^{k+1}}  \frac{\J_j}{x^{j+1}} + \sum_{k ,j> 0}  \J_{-k} x^{k-1} \frac{\J_j}{x^{j+1}} - \sum_{k,j > 0}  \J_{-k-j} x^{k-2} \J_j \right) \tau^{(\bb)}_G 
\end{multline} Applying the operator $\prod_{i=1}^n\nabla(x_i)$, extracting the power of $\hbar^{2g+n}$, and setting $\mathbf{\tilde p} = 0$, we get the following contributions from the three terms in the RHS of equation \eqref{eq:JJterms}. The third term is the most complicated to analyze, so let us start with that:
\begin{align*}
	- [\hbar^{2g+n}]  & \left(\prod_{i=1}^n\nabla(x_i)  \left( \frac{1}{\tau_G^{(\bb)}} \sum_{k ,j > 0} (dx)^2 \J_{-k-j} x^{k-2}  \J_j \tau_G^{{(\bb)}} \right) \right) \Bigg|_{\mathbf{ \tilde p}=0} \\
	&= - (dx)^2 \sum_{i=1}^n \sum_{k,j>0} \frac{(k+j) x^{k-2}}{x_i^k} \sum_{k_1,\cdots,k_{n-1}>0} \frac{F_{g,n}[j,k_1,\cdots,k_{n-1}]}{x_i^{j+1}} \prod_{\substack{1\leq \ell \leq n \\ \ell \neq i}} \frac{dx_\ell }{x^{k_{\ell - \delta_{\ell \geq i}}}_\ell} \\
	&= 	- (dx)^2 \sum_{i=1}^n d_{x_i} \left(\frac{x_i}{x(x-x_i)} \frac{ \phi_{g,n}(x_{[n]})}{dx_i} \right).
\end{align*} In order to get the last line, we use the formula for the sum of a geometric series in the regime $|x|<|x_i|$. The first term gives 
\begin{multline}\label{1st term in JJ}
	[\hbar^{2g+n}]  \left(\prod_{i=1}^n\nabla(x_i)\left(\frac{1}{\tau_G^{{(\bb)}}} \sum_{k > 0} \sum_{j\geq 1} \frac{\J_{k-j} \, dx}{x^{k-j+1}}  \frac{\J_j \, dx}{x^{j+1}} \tau_G^{{(\bb)}}\right)\right)\Bigg|_{\mathbf{ \tilde p}=0}
	 =\\
	   \sum_{\substack{g_1+g_2 = g \\J_1\sqcup J_2=[n]}}  \phi_{g_1,1+|J_1|}(x,x_{J_1})  \phi_{g_2,1+|J_2|}(x,x_{J_2}) 
	+  \phi_{g-1,2+n}(x,x,x_{[n]}),
\end{multline}
while the second term gives
\begin{multline*}
		[\hbar^{2g+n}]  \left(\prod_{i=1}^n\nabla(x_i)\left(\frac{1}{\tau_G^{(\bb)}} \sum_{k > 0} \sum_{j\geq 1}  \J_{-k} x^{k-1} dx \frac{\J_j\, dx}{x^{j+1}} \tau_G^{{(\bb)}}\right)\right)\Bigg|_{\mathbf{ \tilde p}=0}
		 = 	\sum_{i=1}^n \sum_{k>0} \frac{ k x^{k-1} dx dx_i}{x_i^{k+1}}  \phi_{g,n}(x,x_{[n]}\setminus \{x_i\}) 
		\\ = \sum_{i=1}^n\frac{ \phi_{g,n}(x,x_{[n]} \setminus \{x_i\})\, dx dx_i}{(x-x_i)^2}.
\end{multline*}
Using the above computations, equation \eqref{eq:LE1} can be converted into the following form 
\begin{multline*} 
	(t+x)\left( Q^{{\phi}}_{g,1+n}(x;x_{[n]}) \right)
	 +\left(\frac{t(u_1+u_2)}{x} -v\right) dx \,  \phi_{g,1+n}(x,x_{[n]})  +  \frac{tu_1u_2 (dx)^2}{x^2} \delta_{g,0}\delta_{n,0} +  \\
	   \left(\frac{t  \bb}{x} + 2 \bb \right) dx  \phi_{g-\frac12,1+n}(x,x_{[n]}) 
	= (dx)^2 \sum_{i=1}^n d_{x_i} \left(\frac{x_i(t+x_i)}{x(x-x_i)} \frac{ \phi_{g,n}(x_{[n]})}{dx_i}\right).
\end{multline*}  
Finally, we obtain \eqref{formalRLE} after applying the shifts \eqref{eq:phishifts} and dividing by $(t+x)$
\end{proof} 

	It is worth noting that the loop equations of \cref{prop:FLE} are not merely constraints on the behavior of the generating functions $\tilde \phi_{g,1+n}$ (as is often the case for loop equations that one typically encounters in the topological recursion literature). Rather, our loop equations provide an explicit formula for the  $\tilde \phi_{g,1+n}$ in terms of $\tilde \phi_{g',1+n'}$ where $2g'-2+n' < 2g-2+n$\footnote{Such a simplification holds only in restricted cases; for instance, loop equations become more complicated once internal faces are taken into account}. More concretely,
		\begin{corollary}\label{cor:FLE}
		For  $(g,n)\in \frac12\mathbb{Z}_{\geq0} \times \mathbb{Z}_{\geq 0}$ except $(g,n) = (0,0)$, the  generating function $\tilde \phi_{g,1+n}(x,x_1,\ldots,x_n)$ can be expressed as 
		\begin{multline*}
			\tilde \phi_{g,1+n}(x,x_{[n]}) = \\
			- \frac{ \operatorname{Rec}^{\tilde \phi}_{g,1+n}(x;x_{[n]})}{2 \tilde \phi_{0,1}(x)}   
			+ \frac{ 1}{2 \tilde \phi_{0,1}(x)} \left(U_{g,1+n}(x,x_{[n]}) + (dx)^2 \sum_{i=1}^n d_{x_i} \left(\frac{x_i(t+x_i)}{x(t+x)}\frac{\tilde \phi_{g,n}(x_{[n]})}{(x-x_i) dx_i}\right)\right),
			\end{multline*} 
      as formal power series in $x^{-1},x_1^{-1},\ldots, x_n^{-1}$ where we assume that $|x| < |x_i|$ and $\operatorname{Rec}^{\tilde \phi}_{g,1+n}$ is defined in Remark \ref{rem:formal Q}.
	\end{corollary}
	\begin{proof}
		This is a direct corollary of the formal loop equations of \cref{prop:FLE}.
		\end{proof}

\subsection{The spectral curve} \label{sec:spectral curve}
Recall that we fix the weight $G(z)$ as in \eqref{weight}, and let us introduce the associated spectral curve. Set $\Sigma=\mathbb{P}^1_z$, and define two meromorphic functions $x,y$ on it as
\begin{align}  \label{x,y w/o}
	x(z)=t\frac{G(z)}{z},\quad 
	y(z)=\frac{z}{x(z)}.
\end{align} We also define $\tilde y(z) = y(z)+\frac{u_1+u_2}{2x(z)}-\frac{u_1+u_2+v}{2(t+x(z))}$ so that $\tilde y$ becomes anti-invariant under the involution $\sigma$, and the functions $x,\tilde y$  satisfy the following polynomial equation
\begin{equation}
	\left(\tilde y(z)dx(z)\right)^2-U_{0,1}(x(z))=0.\label{curve w/o}
\end{equation}
The set of poles of $\tilde y(z)dx(z)  = \frac{1}{2}\left(y(z) - y(\sigma(z))\right) dx(z)$ can easily be computed as
\begin{equation}\label{eq:P}
	\mathcal{P}=\left\{0,v\right\}\cup\left\{-u_1,-u_2\right\}\cup\left\{-\frac{u_1u_2}{u_1+u_2+v},\infty\right\},
\end{equation}
which are the two preimages of $x=\infty,0,-t$ respectively. 
For convenience\footnote{It does not matter which preimages of $x=0,-t$ we choose to be in $\mathcal{P}_+$ because the associated $\mu_a$ are zero. On the other hand, some of our statements are sensitive to the choice  $0\in\mathcal{P}_+$. If we choose $v$ to be in $\mathcal{P}_+$ instead,  equivalent results would hold with appropriately modified conventions.}, we choose $\mathcal{P}_+=\{0,-u_1,-\frac{u_1u_2}{u_1+u_2+v}\}$ and ${\bm \mu}=\{-1,0,0\}$. These data define a refined spectral curve $\mathcal{S}_{\bm \mu}$. In particular, we have  $\omega_{0,1}(z) = y(z) dx(z)$ and $\omega_{0,2}(z_1,z_2) = -B(z_1,\sigma(z_2))$, and
\begin{equation}
	\omega_{\frac12,1}(z)=\frac{\bb}{2}\left(-\frac{d\tilde y(z)}{\tilde y(z)}-\eta^0(z)\right). \label{1/2,1 w/o}
\end{equation}

Recall that the differentials $\phi_{g,n}$ were defined on a formal disk centered at $x = \infty$. Let us identify this disk with a formal neighborhood of the point $z =0 $ (where $x(z)=\infty$) on the curve $\Sigma$. In the following lemma, we show that the unstable terms $\phi_{0,1}, \phi_{0,2}, \tilde  \phi_{\frac{1}{2},1}$ analytically continue to the meromorphic differentials $\omega_{0,1}, \omega_{0,2}, \omega_{\frac{1}{2},1} $ respectively on the above curve $\Sigma$. 

\begin{lemma}\label{lem:unstable}
	When $(g,n) = (0,1), (0,2)$ or $(\frac{1}{2},1)$, the differentials $\phi_{g,n}$ analytically continue to meromorphic differentials on the refined spectral curve $\mathcal S_{\bm \mu}$ under the identification $x = x(z)$. Moreover, the analytic continuations of $\phi_{0,1} $ and $\phi_{0,2}$ coincide with the  unstable RTR correlators $\omega_{0,1} $ and $\omega_{0,2}$ respectively, while for $(g,n) = (\frac{1}{2},1)$, we have
	\[
		\omega_{\frac{1}{2},1}(z) - \bb dx(z) \left( \frac{1}{2(t+x(z))}+ \frac{1}{2x(z)}  \right) = \phi_{\frac{1}{2},1}(x(z)).
	\] 
\end{lemma}
\begin{proof}
	As refinement does not affect  genus $0$
        correlators, the cases of $(g,n) = (0,1), (0,2)$ are covered
        by \cite{BychkovDuninBarkowskiKazarianShadrin2024}. In particular, the equation satisfied by $ \tilde y dx $ \eqref{curve w/o} agrees with the loop equation of \cref{prop:FLE} when $(g,n) = (0,0)$ under the identification $x = x(z)$. This means that $\tilde{\phi}_{0,1}$ analytically continues to $\tilde y dx(z)$ which is anti-invariant under the involution $\sigma$. The shifts required to obtain $y$ from $\tilde y$ and $\phi_{0,1}$ from $\tilde \phi_{0,1} $ cancel each other, and thus $ \phi_{0,1}$ analytically continues to $ \omega_{0,1} = y dx$. Also,  by definition $\omega_{0,2}(z_1,z_2)=-B(z_1,\sigma(z_2))$ instead of $B(z_1,z_2)$, hence we do not need to shift $\omega_{0,2}$ in our convention, in contrast to \cite{BychkovDuninBarkowskiKazarianShadrin2024}. 
	
	 Let us turn to $\phi_{\frac{1}{2},1}$ now. The loop equation of \cref{prop:FLE} when $ (g,n) = (\frac{1}{2},0)$ reads
	 	\begin{equation*}
	 		2 \tilde  \phi_{0,1}(x) \tilde  \phi_{\frac{1}{2},1}(x) + \bb dx  d_{x} \frac{\tilde \phi_{0,1}(x)}{dx} = - \bb \frac{ v\,(dx)^2}{2x(t+x)}. 
	 	\end{equation*}
	  As we already know that $\tilde \phi_{0,1} $ analytically continues to the curve $\Sigma$ under the identification $x = x(z)$, the above equation implies that $ \tilde  \phi_{\frac{1}{2},1}(x)$ also admits an analytic continuation.
	  
	  Let us now show that this analytic continuation is precisely $\omega_{\frac{1}{2},1}(z)$. First, note that $\eta^0$ (defined in \eqref{eq:etadef}) can be expressed in terms of $x$ and $\tilde y$ as
	\begin{align}
		\eta^0(z )=\frac{v\,dx(z)}{2\tilde y(z)x(z)(t+x(z))}\label{eta0}.
	\end{align}
	One can derive this by considering the pole structure and the anti-invariance with respect to the involution $\sigma$. More concretely, \eqref{curve w/o} implies that the only zeroes of the combination $\tilde y(z)x(z)(t+x(z))$ are of order 1 at ramification points and its only poles are of order 1 at $z=0,v$. In particular, it grows as $\frac{v}{2}x+\mathcal{O}(1)$ as $z\to0$, and this ensures that the RHS of \eqref{eta0} has the same pole structure as $\eta^0$ giving  \eqref{eta0}. Plugging this into the definition of $\omega_{\frac12,1}$ \eqref{eq:1/2def} gives
	\begin{equation*}
		2 \tilde y(z) dx(z)\omega_{\frac12,1}(z)+\bb dx(z)d\tilde y(z)=-\bb\frac{v\,dx(z)dx(z)}{2x(z)(t+x(z))}.
	\end{equation*}
	which matches the loop equation for $ \tilde\phi_{\frac12,1}(x)$ under the identification $x=x(z)$. Expressing $\phi_{\frac{1}{2},1} $ in terms of $\tilde \phi_{\frac{1}{2},1}  $ gives the result.
\end{proof}

\subsection{Refined topological recursion}
Finally, we prove that all  the  differentials $\phi_{g,n}$ analytically continue to the curve $\Sigma$ and, in addition, match the RTR correlators $\omega_{g,n}$ on the refined spectral curve $\mathcal S_{\bm \mu}$ of \cref{sec:spectral curve}. Let us first introduce a useful lemma:
\begin{lemma}\label{lem:kernel} We have
\begin{equation}
    \tilde y(z) dx(z) \frac{\eta^{z_i}(z) }{\tilde y(z_i) dx(z_i) } = \frac{1}{x(z)-x(z_i)}\frac{x(z_i)(t+x(z_i))}{x(z)(t+x(z))}\frac{dx(z)^2}{dx(z_i)}.\label{kernel}
\end{equation} 
\end{lemma}
\begin{proof}
The equation can easily by checked by direct computation. A more analytic approach is similar to the proof of \eqref{eta0}. Namely, as the LHS is invariant under the involution $z\to\sigma(z)$, it is a quadratic differential in $x(z)$ that has poles at the preimages of $x(z)=0,-t, x(z_i),\infty$ which are all simple. With respect to $z_i$, the LHS is a $\sigma$-invariant $(-1)$-form whose only poles are simple at the preimages of $x(z)=x(z_i)$ and ramification points. Then, by considering how the LHS behaves at each pole, e.g. ${\rm LHS}\sim \frac{dx(z)^2}{(x(z)-x(z_i))dx(z_i)}$ as $x(z)\to x(z_i)$, we see that the RHS of \eqref{kernel} is the unique expression with the prescribed behavior.
\end{proof}

\begin{theorem}\label{thm:w/o}
   For every $(g,n) \in \frac12\mathbb{Z}_{\geq0} \times \mathbb{Z}_{\geq1}$, the RTR correlators $\omega_{g,n}(z_1, \ldots,z_n)$ are the analytic continuations of the  generating functions $\phi_{g,n}( x_1,\ldots,x_n)$ to $\Sigma^{n}$ under the identification $x_i = x(z_i)$ (up to the explicit shift for $(g,n) = (\frac{1}{2},1)$ below). In particular, as a series expansion near $z_i = 0$ (where $x(z_i) = \infty$), we have 
   \begin{equation*}
   	\omega_{g,n}(z_1,\ldots, z_n)  - \delta_{g,\frac{1}{2}}\delta_{n,1}  \left(\frac{\bb dx(z_1)  }{2(t+x(z_1))}+\frac{\bb dx(z_1) }{2x(z_1)}\right) = \sum_{\mu_1,\ldots,\mu_n\geq1}F_{g,n}[\mu_1,\ldots,\mu_n]  \prod_{i=1}^n \frac{dx (z_i)}{x(z_i)^{\mu_i+1}}. 
   \end{equation*}

\end{theorem}

\begin{proof}
The unstable cases $(g,n) = (0,1),(0,2),(\frac{1}{2},1)$ are already
covered by \cref{lem:unstable}. 

Let us turn to the stable cases where $ (2g-2+n)>0$. First, we prove by induction on $(2g-2+n)$ that the generating functions $\phi_{g,n}$ analytically continue to meromorphic differentials on the  curve $\Sigma $ under the identification $x = x(z)$. Indeed, \cref{cor:FLE} provides an expression for  $\tilde \phi_{g,n}$ in terms of the correlators  $\tilde \phi_{g',n'}$ with $(2g'-2+n') < (2g-2+n)$. By the induction hypothesis, we see that the stable $\tilde \phi_{g,n}(x_1,\ldots,x_n) = \phi_{g,n}(x_1,\ldots,x_n) $ also admit an analytic continuation to the  curve $\Sigma$. Let us denote these analytically continued differentials by $\check{\phi}_{g,n}(z_1,\ldots,z_n)$.

For any stable $(g,n) \in  \frac{1}{2} \mathbb Z_{\geq 0} \times \mathbb Z_{\geq 0}$, the formal loop equations of \cref{prop:FLE} imply that the $\check{\phi}_{g,1+n}(z,z_1,\ldots,z_n)$ are given by the following explicit formula:
\begin{multline} \label{eq:checkLE}
	 \check \phi_{g,1+n} (z,z_{[n]})   = -\frac{\operatorname{Rec}_{g,1+n}^{\check \phi}(z;z_{[n]})}{2\check \phi_{0,1}(z)} - \frac{U_{g,n+1}(x(z),x(z_{[n]}))}{2\check \phi_{0,1}(z)} \\ + \frac{(dx(z))^2}{2\check \phi_{0,1}(z)} \sum_{i=1}^n d_{z_i}  \left(\frac{x(z_i)(t+x(z_i))}{x(z)(t+x(z))}\frac{\check \phi_{g,n}(z_{[n]})}{(x(z)-x(z_i)) dx(z_i)}\right),
\end{multline} where we use the notation $\operatorname{Rec}_{g,1+n}^{\check \phi}$ as defined in \eqref{eq:Recdef} with the choice of the required meromorphic function to be $x(z)$. Note that $\check \phi_{0,1}(z)$ is anti-invariant under the involution $\sigma$.

On the other hand, the refined TR correlator $\omega_{g,1+n}$ can be expressed in terms of $\omega_{g',1+n'}$ such that $2g'-2+1+n'< 2g-2+1+n$: 
\begin{equation}\label{eq:omegagn}
    \omega_{g,1+n}(z,z_{[n]})=-\frac{\text{Rec}^\omega_{g,1+n}(z;z_{[n]})}{2\tilde y(z) dx(z)} - \frac{U_{g,1+n}(x(z),x(z_{[n]}))}{2 \tilde y(z)d x(z)} +\sum_{i=1}^nd_{z_i}\left(\frac{\eta^{z_i}(z)}{2\tilde y(z_i)dx(z_i)}\omega_{g,n}(z_{[n]})\right).
\end{equation} This statement holds for all stable $\omega_{g,1+n}$ such that $2g-2+1+n > 0$ (except $(g,1+n) = (0,3), (1,1)$) due to \cite[equation (A.15)]{KidwaiOsuga2023}. The case  $(g,1+n) = (0,3)$ is proved in equation (A.13) of \textit{loc. cit.}, and the last remaining case $(g,1+n) = (1,1)$ follows from equation (A.14) of \textit{loc. cit.} after evaluating the residues at the points $ p \in \mathcal P_+ $ explicitly\footnote{The definition of $\mathcal{P}_+$ in the present article is $\widetilde{\mathcal{P}}_-$ in \cite{KidwaiOsuga2023}.}. 

After substituting the formula \eqref{kernel} in Lemma \ref{lem:kernel} into equation \eqref{eq:omegagn}, we observe that equation \eqref{eq:omegagn}
takes exactly the same form as  the formula \eqref{eq:checkLE} for $\check \phi_{g,1+n}(z,z_{n})$. Finally, by induction on $(2g-2+n)$, we see that $ \check \phi_{g,n}$ analytically continues to the $\omega_{g,n}$, which finishes the proof.
\end{proof}

\section{Extension: inserting internal faces}
\label{sec:Internal}

In this section, we consider $G$-weighted $\bb$-Hurwitz numbers (with $G (z)= \frac{(u_1+z)(u_2+z)}{v-z} $ fixed as before) with internal faces, and prove that refined topological recursion can be used to compute  them. 

\subsection{Inserting internal faces}
Consider the correlators $\omega_{g,n}$ computed by refined topological recursion on the refined spectral curve $\mathcal S_{\bm \mu}$ of \cref{sec:spectral curve}. Recall that these correlators were proved to compute $ G$-weighted $\bb$-Hurwitz numbers in \cref{thm:w/o}. Now, for $D,m\in\mathbb{Z}_{\geq1}$, consider a tuple of parameters $(p_1,\ldots,p_D)$ and  define\footnote{There is a  missing minus sign in the analogous equation appearing in \cite[Definition 3.3]{BonzomChapuyCharbonnierGarcia-Failde2024}.}
\begin{equation}\label{eq:FgnD}
    F_{g,n}^{D,m}\left[k_1,\ldots,k_n\right]:=(-1)^{n+m}\Res_{z_1=0}\cdots\Res_{z_{n+m}=0}\left(\prod_{i=1}^nx(z_i)^{k_i}\right)\left(\prod_{j=n+1}^{n+m}V(x(z_{j}))\right)\; \omega_{g,n+m}(z_1,\ldots,z_{n+m})
\end{equation}
where we define  the \emph{potential} $V$  as
\begin{equation*}
    V(x):=\sum_{i=1}^{D}\frac{p_i}{i}x^{i}.
\end{equation*}
By introducing the degree counting parameter $\epsilon$, we further define a series in $\epsilon$ by
\begin{align}
  \label{eq:FInternal}
    F_{g,n}^{D}\left[k_1,\ldots,k_n;\epsilon\right]:= F_{g,n}\left[k_1,\ldots,k_n\right] + \sum_{m \geq 1}\frac{\epsilon^m}{m!}F_{g,n}^{D,m}\left[k_1,\ldots,k_n\right].
\end{align}
We claim that $F_{g,n}^{D}\left[k_1,\ldots,k_n;\epsilon\right]$ in
\eqref{eq:FInternal} and
$F_{g,n}^{D}\left[k_1,\ldots,k_n;\epsilon\right]$ introduced in
\cref{prop:HurwitzWithBoundaries} is the same object, so that
it is (up to the combinatorial factor given in
\eqref{eq:HurwitzWithBoundaries}) the generating
series of  colored monotone Hurwitz
maps of genus $g$ with $n$ marked boundaries of degrees
$k_1,\ldots,k_n$ and with internal faces of degree $i$ weighted by
$\epsilon t^i p_i$. Indeed, the RHS of~\eqref{eq:FgnD} can be rewritten as
\[
  \sum_{k_{n+1},\dots,k_{n+m}=1}^D\prod_{i=n+1}^{m+n}\frac{p_{k_i}}{k_i}
  (-1)^{n+m}\Res_{z_1=0}\cdots\Res_{z_{n+m}=0}\left(\prod_{i=1}^{n+m}x(z_i)^{k_i}\right)\,
  \omega_{g,n+m}(z_1,\ldots,z_{n+m}),\]
which, by~\cref{thm:w/o}, is equal to $F_{g,n}^{D}\left[k_1,\ldots,k_n;\epsilon\right]$ introduced in
\cref{prop:HurwitzWithBoundaries}. In particular, the shift for $(g,n) = (\frac{1}{2},1)$ appearing in \cref{thm:w/o} does not contribute here as $m \geq 1$, and thus $m+n \geq 2$.

\begin{definition}
Consider  the generating function of $F_{g,n}^{D}\left[k_1,\ldots,k_n;\epsilon\right]$, for $(g,n) \in \frac{1}{2}\mathbb Z_{\geq 0} \times \mathbb Z_{\geq 1}$ given by
\begin{equation*}
 \phi_{g,n}^D(X_1,\ldots,X_n) := \sum_{k_1,\ldots, k_n \geq 1} F_{g,n}^{D}\left[k_1,\ldots,k_n;\epsilon\right] \prod_{i=1}^n \frac{ d X_i}{X_i^{k_i+1}}  + \delta_{g,0}\delta_{n,1} \epsilon V'(X_1)dX_1.
\end{equation*}  
\end{definition}
 Again these $\phi_{g,n}^D$  are germs of meromorphic differentials on the $n$-th product of formal disks with coordinates $X_i^{-1}$. The goal of this section is to prove that the $\phi^D_{g,n}$ analytically continue to meromorphic differentials on an algebraic curve, and moreover that these analytic continuations coincide with the RTR correlators on a certain refined spectral curve (up to an explicit shift for $(g,n) = (\frac{1}{2},1)$). This result will simultaneously prove that $F_{g,n}^{D}\left[k_1,\ldots,k_n;\epsilon\right]$ is a convergent series in $\epsilon$.

\subsection{The (unrefined) spectral curve} 

When $\bb = 0$,  \cite{BonzomChapuyCharbonnierGarcia-Failde2024} and  \cite{BychkovDuninBarkowskiKazarianShadrin2025} prove independently that TR on a certain spectral curve computes Hurwitz numbers with internal faces. Let us briefly review these results following the presentation of \cite{BonzomChapuyCharbonnierGarcia-Failde2024}. We only focus on the statements that are relevant for this paper.

Consider $\Sigma=\mathbb{P}^1$ with the global coordinate $z$ as before. Let $\{A^{(i)}\}_{i\in\{1,2,3\}}$ and $\{B^{(i)}\}_{i\in\{1,2,3\}}$ be polynomials in $z$ and $z^{-1}$ respectively, which are determined by the following relations between them
\begin{align*}
    A^{(1)}(z)=&1+\frac{u_2 t}{v} \left\{z\frac{B^{(2)}(z)}{B^{(3)}(z)}\right\}^{\geq0}, &B^{(1)}(z)=&1+\frac{\epsilon}{u_1} \sum_{s=1}^D p_s\left[\frac{A^{(1)}(z)^{s-1}A^{(2)}(z)^s}{A^{(3)}(z)^s}\right]^{<0},\nonumber\\
    A^{(2)}(z)=&1+\frac{u_1 t}{v} \left\{z\frac{B^{(1)}(z)}{B^{(3)}(z)}\right\}^{\geq0}, &B^{(2)}(z)=&1+\frac{\epsilon}{u_2} \sum_{s=1}^D p_s\left[\frac{A^{(1)}(z)^{s}A^{(2)}(z)^{s-1}}{A^{(3)}(z)^s}\right]^{<0},\nonumber\\
    A^{(3)}(z)=&1-\frac{tu_1u_2}{v^2} \left\{z\frac{B^{(1)}(z)B^{(2)}(z)}{B^{(3)}(z)^2}\right\}^{\geq0},&B^{(3)}(z)=&1-\frac{\epsilon }{v}\sum_{s=1}^D p_s\left[\frac{A^{(1)}(z)^{s}A^{(2)}(z)^s}{A^{(3)}(z)^{s+1}}\right]^{<0},\nonumber\\
\end{align*}
where for $f(z)=\sum_{k\geq d}f_kz^k$ and $g(z)=\sum_{k\geq d}g_kz^{-k}$ we define $[f]^{<0}=\sum_{k=d}^{-1}f_kz^k$ and $\{g\}^{\geq0}=\sum_{k=d}^0g_kz^{-k}$.\footnote{Note that the dictionary between our parameters and the parameters in \cite{BonzomChapuyCharbonnierGarcia-Failde2024} is
\[
    u_1^{\text{BCCG}}=\frac{1}{u_1},\quad u_2^{\text{BCCG}}=\frac{1}{u_2},\quad u_3^{\text{BCCG}}=-\frac{1}{v},\quad t^{\text{BCCG}}= \frac{u_1 u_2t}{v},\quad\alpha^{\text{BCCG}}=\epsilon.
\]} 
Then, we define the meromorphic functions $X,  Y$ on $\Sigma$ as
\begin{align}\label{eq:Xdef}
    X(z):=\frac{A^{(1)}(z)A^{(2)}(z)}{zA^{(3)}(z)},\qquad Y(z):=\frac{u_1-u_1A^{(1)}(z)B^{(1)}(z)}{X(z)}.
\end{align}
The involution $\sigma:\Sigma\to\Sigma$ is given by the non-trivial solution of $X(z)=X(\sigma(z))$. Furthermore, recall that $B$ is the fundamental bidifferential on $\Sigma$, and  define the unstable correlators 
\begin{equation*}
    \omega_{0,1}^{D}(z):= Y(z)dX(z),\qquad \omega_{0,2}^{D}(z_1,z_2):=-B(z_1,\sigma(z_2)),
\end{equation*} to complete the definition of the (unrefined) spectral curve. 

Then one of the main results of \cite{BonzomChapuyCharbonnierGarcia-Failde2024, BychkovDuninBarkowskiKazarianShadrin2025} (which works in a much more general context than we are considering in this paper, namely for double weighted Hurwitz numbers with the weight being the product of an exponential and a rational function) is that the Chekhov--Eynard--Orantin topological recursion correlators are generating functions for weighted Hurwitz numbers with internal faces for any $(g,n) \in \mathbb Z_{\geq 0}\times \mathbb Z_{\geq 1}$. The statements that we will use in this paper only concern the unstable correlators, so let us recall those.

\begin{theorem}[\cite{BonzomChapuyCharbonnierGarcia-Failde2024, BychkovDuninBarkowskiKazarianShadrin2025}]
	The differentials $ \phi^D_{0,1}$ and $ \phi^D_{0,2}$ viewed as meromorphic differentials on a formal neighborhood of $z= 0$ (where $X(z) = \infty$) on $\Sigma$ analytically continue to meromorphic differentials on the whole Riemann surface $\Sigma$ under the identification $X = X(z)$. Moreover, these analytic continuations coincide with the unstable TR correlators $\omega_{0,1}^D$ and $\omega_{0,2}^D$ respectively.
\end{theorem}

As part of their proof, the authors of \cite{BonzomChapuyCharbonnierGarcia-Failde2024} also prove a variational formula for the TR correlators w.r.t the parameter $\epsilon$, and this will be crucial for our computations. 
\begin{theorem}\label{thm:unstablevariation}
	The following variational formulas hold for the correlators $\omega_{0,1}$ and $\omega_{0,2}$:
	\begin{align} \nonumber
		\delta_{\epsilon}^{(1)} \omega_{0,1}^{D}(z_0)&=-\Res_{z=0}V(X(z))\left(\omega_{0,2}^{D}(z,z_0)+\frac{dX(z)dX(z_1)}{(X(z)-X(z_1))^2}\right),\\
		\delta_{\epsilon}^{(2)}\omega_{0,2}^{D}(z_0,z_1)&=-\Res_{z=0}V(X(z))\omega_{0,3}^{D}(z,z_0,z_1).\label{variation of w_{0,2}}
	\end{align}
\end{theorem}

Note that for \eqref{variation of w_{0,2}}, $\omega_{0,2}^D(z_1,z_2)$ differs from $\omega_{0,2}^{{\rm BCCG}}(z_1,z_2)$ in \cite{BonzomChapuyCharbonnierGarcia-Failde2024}, but the difference is $\frac{dX(z_1)dX(z_2)}{(X(z_1)-X(z_2)^2}$ which vanishes after acting by the variational operator. One of the key ideas of the proof of \cite{BonzomChapuyCharbonnierGarcia-Failde2024} is that inserting internal faces as in \eqref{eq:FgnD} is very closely related to the variational formula  of \cref{thm:unstablevariation} with the choice $\Lambda  = - V(X)$. We will import this idea to the setting of $\bb$-Hurwitz numbers to deduce the correct refined spectral curve.

\subsection{Refined spectral curve with internal faces.}
\label{sub:SpectralInternal}
 In this section, we will find a refined spectral curve satisfying the deformation condition of \cref{def:deformation} with respect to $\epsilon$. Then, in the next section, we will use the variational formula to prove that the correlators on this refined spectral curve are precisely the analytic continuations of $\phi^D_{g,n}$ to the spectral curve.

\noindent Let us define
\begin{equation}\label{eq:Ydef}
    \tilde Y(z):= Y(z)+\frac{t(u_1+u_2)-v X(z)}{2X(z)(t+X(z))}-\epsilon \frac{V'(X(z))}{2}.
\end{equation}
This shift makes $\tilde Y(z)$ anti-invariant under the involution $\sigma$ as shown below.

\begin{lemma}\label{lem:Y^2}
	For a generic choice of $(\epsilon,p_1,\dots,p_D)$, there exists a polynomial $M(X)$ of degree $D$ such that $M(X)|_{\epsilon=0}=1$ and
	\begin{equation}
		\tilde Y^2=\frac{(v X-a)(v X-b)M(X)^2}{4X^2(t+X)^2}.\label{curve w/}
	\end{equation}
	Here, generic means that we avoid a set of values of $(\epsilon,p_1,\dots,p_D)$ where $a$ coincides with $b$ or where roots of $M(X)$ coincide with $\frac{a}{v}$, $\frac{b}{v}$, or each other.
\end{lemma}
\begin{proof}
	Let us consider the loop equation for $ \phi_{0,1}^D$, or equivalently $ Y$, given by
	\begin{equation*}
	  [\hbar^{1}]  \left(\frac{1}{\tau_G^{(\bb)}}\sum_{k\geq0}\frac{D_k (dx)^2 }{x^{k+2}} \tau_G^{(\bb)} \right)\Bigg|_{\tilde p_{k> D}=0}=0.
	\end{equation*}
     The only difference from \eqref{eq:LE1} is that we set $\tilde p_k=0$ only for $k>D$, hence one can mostly repeat computations in Lemma \ref{prop:FLE}. After some manipulation, we find:
       \begin{multline} \label{LE:(0,1) with}
        (t+X)  Y(X)^2+\left(\frac{t(u_1+u_2)}{X}-v-\epsilon(t+X)V'(X)\right) Y(X) \\
        + \frac{tu_1 u_2}{X^2} -\epsilon\left(\frac{t(u_1+u_2)}{X}-v\right)V'(X)+\epsilon\left[ (t+X)V'(X) Y(X)\right]_{\leq -2} = 0
       \end{multline}
      where for $f =
      \sum_{n=-\infty}^q a_n x^{n}$, we define $ [f]_{\leq p} := \sum_{n=-\infty}^p a_n x^{p}$. Although the last term depends on $ Y(X)$, one can show by degree counting of \eqref{LE:(0,1) with} and by shifting $ Y$ to $ \tilde Y$ that there exists a polynomial $P^D(X)$ of degree $2D+2$ such that $\tilde Y^2=\frac{P^D(X)}{4X^2(t+X)^2}$. 
      
      One can further constrain the form of $P^D(X)$ as follows. First, $\tilde Y(X)$ admits an analytic continuation to $\Sigma$ as $\tilde Y$ and $X$ satisfy an algebraic equation. Next, since $\tilde Y$ is a meromorphic function on $\Sigma=\mathbb{P}^1$, $P^D(X)$ can have at most two simple zeroes. The other zeroes must be double zeroes, i.e. $P^D(X)=(vX-a)(vX-b)M(X)^2$ --- otherwise $\tilde Y$ would be a function on a higher-genus curve. Finally, in the limit $\epsilon\to0$ while all $p_1,\dots,p_D$ are fixed, $\frac{P^D(X)(dX)^2}{4X^2(t+X)^2}$ should reduce to  $U_{0,1}(x)$ in \eqref{U_{0,1}}, hence we have $M(X)|_{\epsilon=0}=1$. Note that  $(\epsilon,p_1,\dots,p_D)$ is generic, so  all the roots of $M(X)$ are distinct and $a$ does not coincide with $b$.
          \end{proof}

\noindent 
From the form of equation \eqref{curve w/} for $\tilde Y(z)$, we see that the set $\mathcal P$ of poles and zeroes (excluding ramification points)  of the differential $(Y(z)-Y(\sigma(z)) dX(z)$ consists of the preimages of $X=\infty,0,-t$ and the preimages of all $D$ (generically distinct) roots of $M(X)$. Let us choose a splitting $\mathcal P_+ \sqcup \sigma(\mathcal P_+) = \mathcal P$ as follows. From the definition of $X(z)$ in equation \eqref{eq:Xdef}, we know that $z = 0$ is one of the preimages of $X(z) = \infty$, and we choose 
$z= 0$ to be in $\mathcal P_+$. In the limit $\epsilon \to 0$, we know that two preimages $a_1,a_1'$ of $X(z)=0$ tend to $a_1 = -u_1 + O(\epsilon)$, $a'_1 = -u_2 + O(\epsilon)$ respectively (see equation \eqref{eq:P}), and we choose $a_1 \in \mathcal P_+$. Similarly, from the preimages of $X(z) = -t $, we choose  $a_2 \in \mathcal P_+$ satisfying $a_2 = -\frac{u_1u_2}{u_1+u_2+v} + O(\epsilon)$.

Since $M(X)|_{\epsilon=0}=1$, all the roots of $M(X)$ must tend to $\infty$ as $\epsilon\to0$. This implies that for each root, one of the preimages  $b_i$ behaves as $b_i= O(\epsilon)$ in the limit $\epsilon\to0$ and we choose such $b_1,\ldots,b_D$ as elements of $\mathcal{P}_+$. With this choice, we are ready to define the refined spectral curve that we are interested in.
\begin{definition}\label{def:refinedSCint}
	We define the \textit{refined spectral curve} $\mathcal{S}_{\bm \mu}^{D}$ as the collection 
	\begin{itemize}
		\item $(\Sigma = \mathbb P^1_z,X,Y)$, where $X $ and $Y$ are defined in equations \eqref{eq:Xdef} and $\eqref{eq:Ydef}$ respectively;
		\item and, $(\mathcal P_+, \{\mu_a\}_{a\in \mathcal P_+})$ as 
		\begin{equation*}
			\setlength\arraycolsep{1pt}
			\begin{matrix}
				\mathcal{P}_+ & = & \{& 0,&a_1,&a_2,& b_1,&\ldots,& b_D&\}\\
				{\bm \mu} & = & \{&\mu_0,&\mu_{a_1},&\mu_{a_2}, &1,&\ldots,&1&\},
			\end{matrix}
		\end{equation*}
		where $\mu_0,\mu_{a_1}\,\mu_{a_2}\in\mathbb{C}$.
	\end{itemize} We  use  $\omega_{g,n}^D$ to denote the correlators constructed by refined topological recursion on $\mathcal{S}_{\bm \mu}^{D}$.
\end{definition} 

For convenience, we define $\tilde \omega^D_{0,1} : = \tilde Y(z) d X(z)$, which only differs from $\omega^D_{0,1}(z) = Y(z) dX(z)$ by a differential in $X(z)$.   \cref{lem:Y^2} implies that we can write $2 \tilde\omega^D_{0,1}(z) = \omega^D_{0,1}(z) - \omega^D_{0,1}(\sigma(z)) $. Let us now prove the following useful lemma that gives a different expression for $\omega^D_{\frac12,1}$ that is more suited for understanding its variations with respect to $\epsilon$.
\begin{lemma}\label{lem:1/21}
	The correlator $\omega^D_{\frac12,1}$ can be expressed as
	\begin{multline} \label{1/2,1 as as integral}
		\omega^D_{\frac{1}{2},1} (z_0)= \left( \frac{\bb}{2\pi \rm i} \oint_{z\in C_+} \frac{\eta^z(z_0)}{2 \tilde \omega_{0,1}^{D}(z)}dX(z)d\tilde Y(z) \right)-\frac{\bb}{2}(\mu_0-\kappa_0)\Res_{z=0}\eta^z(z_0)\frac{dX(z)}{X(z)} \\  +\frac{\bb}{2}\sum_{i=1}^2(\mu_{a_i}-\kappa_{a_i})\Res_{z=a_i}\eta^z(z_0)\frac{dX(z)}{X(z)-X(a_i)},
	\end{multline} where
	$\kappa_0 = 1-D$ and $\kappa_{a_1} = \kappa_{a_2} = -1$.
\end{lemma}
\begin{proof}
	We start with the definition of $\omega^D_{\frac{1}{2},1} $ and rewrite as 
\begin{align*} \nonumber 
	\omega^D_{\frac{1}{2},1} (z_0) &= \frac{\bb}{2}\left( -\frac{d\tilde Y(z_0)}{\tilde Y(z_0)} + \sum_{c \in \mathcal P_+} \mu_c \eta^c(z_0) \right) \\ \nonumber
	&=  \frac{\bb}{2}\left( \Res_{z=z_0}\frac{\eta^z(z_0)}{\tilde \omega_{0,1}^{D}(z)}dX(z)d\tilde Y(z) + \sum_{c \in \mathcal P_+} \mu_c \eta^c(z_0) \right) \\ \nonumber
	&=  \frac{\bb}{2}\left( \Res_{z=z_0, \mathcal P_+}\frac{\eta^z(z_0)}{\tilde \omega_{0,1}^{D}(z)}dX(z)d\tilde Y(z) + \sum_{c \in \mathcal P_+} (\mu_c-\kappa_c) \eta^c(z_0) \right) \\ \nonumber
	&= \left( \frac{\bb}{2\pi \rm i} \oint_{z\in C_+} \frac{\eta^z(z_0)}{2\tilde \omega_{0,1}^{D}(z)}dX(z)d\tilde Y(z) \right) + \frac{\bb}{2} \sum_{c = 0,a_1,a_2} (\mu_c-\kappa_c) \eta^c(z_0) \\  
	&= \left( \frac{\bb}{2\pi \rm i} \oint_{z\in C_+} \frac{\eta^z(z_0)}{2\tilde \omega_{0,1}^{D}(z)}dX(z)d\tilde Y(z) \right)-\frac{\bb}{2}(\mu_0-\kappa_0)\Res_{z=0}\eta^z(z_0)\frac{dX(z)}{X(z)} \\ \nonumber
	&\hspace{5cm} +\frac{\bb}{2}\sum_{i=1}^2(\mu_{a_i}-\kappa_{a_i})\Res_{z=a_i}\eta^z(z_0)\frac{dX(z)}{X(z)-X(a_i)},
\end{align*} where we use the property that $\eta^z(z_0)$ is a differential in $z_0$ with residue $\pm1$ at $z_0 = z, \sigma(z)$ to get the second line,  the form of $\tilde Y$ given by \eqref{curve w/} to  evaluate the residues at $z = \mathcal P_+$ and get the third line, and the form of $X$ \eqref{eq:Xdef} near $z = 0,a_1,a_2$ to get the last equality. We define $\kappa_c$ for $c \in \mathcal P_+$ as
\[
\left(\kappa_0,\kappa_{a_1},\kappa_{a_2},\kappa_{b_1},\cdots, \kappa_{b_D}\right)= \left(1-D,-1,-1,1,\ldots,1\right),
\]    and notice that we have used the specific choice of the parameters $\mu_{b_i} = 1$ to remove the terms corresponding to $c= b_i$ and get equation \eqref{1/2,1 as as integral}.
\end{proof}

\noindent Now we show that  the refined spectral curve $\mathcal{S}_{\bm \mu}^{D}$  satisfies the refined deformation condition.
\begin{proposition}\label{prop:deform}
    For any choice of $\mu_0,\mu_{a_1},\mu_{a_2} \in \mathbb C$, the refined spectral curve $\mathcal{S}_{\bm \mu}^{D}$ satisfies the refined deformation condition (\cref{def:deformation}) with respect to $\epsilon$, where we choose  the  contour $\gamma$  to be a small circle around $z = 0$ and the function $\Lambda$ to be $\Lambda =- V(X(z))$.
\end{proposition}
\begin{proof}
	We need to check Conditions $\textbf{D1}$ and $\textbf{D2}$ of Definition \ref{def:deformation}.
    First, the function $\Lambda(z)=-V(X(z))$ only has poles at the preimages of $X=\infty$, and the poles are of order $D$. Since $\omega_{0,1}^{D} \sim X^{D-1} dX$ near $X = \infty$, it has a pole of order $D+1$. Thus, Condition $\textbf{D1}$ of Definition \ref{def:deformation} holds.
    
    As for condition $\textbf{D2}$, the required variations of $\omega_{0,1}$ and $\omega_{0,2}$ are known thanks to the results of \cite{BonzomChapuyCharbonnierGarcia-Failde2024} (see \cref{thm:variation}).  Thus, all we need to check is the following  formula for the variation of $\omega_{\frac{1}{2},1}$ w.r.t $\epsilon$:
    \begin{equation}\label{eq:toshow}
    	\delta_\epsilon^{(1)}\omega^D_{\frac12,1}(z_0)= - \Res_{z=0} \Lambda(z)\,\omega^D_{\frac12,2}(z,z_0).
    \end{equation}

We will apply the variational operator $\delta^{(1)}_{\epsilon}$ to  the expression for $\omega^D_{\frac12,1}$ derived in equation \eqref{1/2,1 as as integral}. For the first term in the RHS, we can apply \cite[Lemma A.3]{Osuga2024} to get
\begin{multline*}
    \delta^{(1)}_{\epsilon}\left(\frac{\bb}{2 \pi \rm i}\oint_{z\in C_+}\frac{\eta^z(z_0)}{2\tilde \omega_{0,1}^{D}(z)}dX(z)d\tilde Y(z)\right) \\
    =\frac{\bb}{2 \pi \rm i}\oint_{z\in C_+}\frac{\eta^z(z_0)}{2\tilde \omega_{0,1}^{D}(z)}\left(dX(z)d\delta^{(1)}_{\epsilon}\tilde Y(z)+\frac{2\delta_\epsilon^{(1)}\tilde \omega_{0,1}^{D}(z)}{2 \pi \rm i}\oint_{\check z\in C_+}\frac{\eta^{\check z}(z)}{2\tilde \omega_{0,1}^{D}(\check z)}dX(\check z)d\tilde Y(\check z)\right),
\end{multline*} where the first term comes from the variation of $\tilde Y(z)$ in the numerator and the second term comes from the variation of $\frac{\eta^z(z_0)}{2\tilde \omega^D_{0,1}(z)}$. As for the second term in the RHS of \eqref{1/2,1 as as integral} which corresponds to $c=0$, we have
\begin{align*}
    \delta^{(1)}_{\epsilon}\left(\Res_{z=0}\eta^z(z_0)\frac{dX(z)}{X(z)}\right)&=\Res_{z=0}\frac{dX(z)}{X(z)}\delta_\epsilon^{(2)}\eta^z(z_0) \\
    &=\frac{1}{2 \pi \rm i}\oint_{z\in C_+}\frac{\eta^z(z_0)}{2\tilde \omega_{0,1}^{D}(z)} \left( 2\delta_\epsilon^{(1)} \tilde \omega_{0,1}^{D}(z) \Res_{\check z=0}\eta^{\check z}(z)\frac{dX(\check z)}{X(\check z)} \right),
\end{align*}
where, in the first line $\delta^{(1)}_{\epsilon}$ commutes with the residue as the contour can be chosen independent of $\epsilon$. To get the second equality, we follow the same  technique as the proof of \cite[Lemma A.3]{Osuga2024}. Analogous results hold for the contributions from $c=a_1,a_2$. All together, we find
\begin{equation}\label{eq:1/21var}
    \delta_\epsilon^{(1)}\omega_{\frac12,1}^{D}(z_0)=\frac{1}{2\pi \rm i}\int_{z\in C_+} \frac{\eta^z(z_0)}{2\tilde \omega^D_{0,1}(z)}\left(2\omega^D_{\frac12,1}(z)\;\delta_\epsilon^{(1)} \tilde \omega^D_{0,1}(z)+\bb dX(z)d\delta_\epsilon^{(1)}\tilde Y(z)\right),
\end{equation} where we have obtained $\omega^D_{\frac12,1}(z)$ inside the contour integral by using the expression \eqref{1/2,1 as as integral}.

Let us compare this with the form of $\omega_{\frac12,2}^D$. We take the definition of $\omega_{\frac12,2}^D$ as given by the refined topological recursion formula in \eqref{eq:RTRequiv} and rewrite as
\begin{multline}\label{eq:1/22equiv}
	 \omega^D_{\frac12,2}(z_0,z_1)
	 =\frac{1}{2\pi \rm i}\oint_{z\in C_+} \frac{\eta^z(z_0)}{2\tilde \omega_{0,1}^{D}(z)} \left( \omega^D_{\frac12,1}(z) \left(\omega^D_{0,2}(z,z_1)-\omega^D_{0,2}(\sigma(z),z_1)\right) 
	 \right.
	\\ + \left. \bb dX(z)d_z\frac{\omega^D_{0,2}(z,z_1)-\omega^D_{0,2}(\sigma(z),z_1)}{2dX(z)}\right),
\end{multline} 
using the following formula
    \begin{equation*}
        \oint_{z\in C_+}\frac{\eta^z(z_0)}{2\tilde \omega_{0,1}^{D}(z)}\bb dX(z)d_z\left(\frac{dX(z_1)}{(X(z)-X(z_1))^2}\right)=0,
    \end{equation*} which holds as integrating along $z\in C_+$ gives the same contribution as integrating along $z \in C_-$, but on the other hand, integrating along $z \in C_+ \cup C_-$ gives zero as it contains all the poles of the integrand.

Finally, we  compute the RHS of equation \eqref{eq:toshow} using  the expression for $\omega_{\frac12,2}$ derived in \eqref{eq:1/22equiv}: 
    \begin{equation*}
       - \Res_{z_1=0}\Lambda(z_1) \omega_{\frac12,2}(z_0,z_1)=\frac{1}{2\pi i}\int_{z\in C_+} \frac{\eta^z(z_0)}{2\tilde \omega^D_{0,1}(z)}\left(2\omega^D_{\frac12,1}(z)\;\delta_\epsilon^{(1)}\tilde \omega^D_{0,1}(z)+\bb dX(z)d\delta_\epsilon^{(1)}\tilde Y(z)\right),
    \end{equation*}  where we use the fact that both $z=0$ and $z=z_1$ are contained in $C_+$ to commute the contour integral with respect to $z$ and the residue at $z_1=0$. The RHS of this equation matches the RHS of equation \eqref{eq:1/21var} which proves the variational formula \eqref{eq:toshow} for $\omega^D_{\frac12,1}$  and thus the lemma.
\end{proof}

\subsection{Refined topological recursion with internal faces}

We need two more lemmas to prove that the refined topological recursion correlators $\omega_{g,n}^D$ are the analytic continuations of the $\phi^D_{g,n}$  to the refined spectral curve $\mathcal S^D_{\bm \mu}$.  

\begin{lemma}\label{lem:alpha-derivative}
    Let $\{P_i\}_{i\in\{1,\ldots,n\}}$ be a set of polynomials whose coefficients do not depend on $\epsilon$. Then, for a symmetric meromorphic $n$-differential $\omega$ on $\mathbb P^1_z$, we have
    \begin{multline*}
     \frac{\partial}{\partial\epsilon} \left(\Res_{z_1=0}\cdots\Res_{z_n=0}\, P_1(X(z_1))\cdots P_n(X(z_n))\,\omega(z_1,\ldots,z_n)\right)\\
     =\Res_{z_1=0}\cdots\Res_{z_n=0}\, P_1(X(z_1))\cdots P_n(X(z_n))\,\delta_\epsilon^{(n)}\,\omega(z_1,\ldots,z_n).
\end{multline*}

\end{lemma}

\begin{proof}
    When $n=1$, let $\mathcal{W}$ be the associated function of $\omega$, i.e. $\omega(z)=\mathcal{W}(z)dX(z)$. Then, 
\begin{align}
     \frac{\partial}{\partial\epsilon}& \Res_{z=0}P(X(z))\omega(z) = \Res_{z=0}\frac{\partial}{\partial\epsilon} \Big(P(X(z))\mathcal{W}(z)dX(z)\Big)\nonumber\\
   &= \Res_{z=0}\left(\frac{\partial X(z)}{\partial\epsilon}\frac{\partial P(X(z))}{\partial X(z)}\mathcal{W}(z)dX(z)+P(X(z))\frac{\partial\mathcal{W}(z)}{\partial\epsilon}dX(z)+P(X(z))\mathcal{W}(z)d\frac{\partial X(z)}{\partial\epsilon}\right)\nonumber\\
   &=\Res_{z=0}\left(d\left(P(X(z))\mathcal{W}(z)\frac{\partial X(z)}{\partial\epsilon}\right)+P(X(z))\left(\frac{\partial\mathcal{W}(z)}{\partial\epsilon}-\frac{\partial X(z)}{\partial\epsilon}\frac{d\mathcal{W}(z)}{dX(z)}\right)dX(z)\right)\nonumber\\
   &=\Res_{z=0}P(X(z))\;\delta_\epsilon^{(1)}\,\omega(z)\nonumber
\end{align}
The first equality holds because one can always choose a contour encircling $z=0$ independent of $\epsilon$. The second equality is due to the assumption that $P(X)$ does not depend on $\epsilon$. The fourth equality holds because the first term is the total derivative of a meromorphic function. One can repeat the same argument when $n\geq2$.
\end{proof}
 The next lemma shows that the  correlators $\omega^D_{g,n}$ on the spectral curve $\mathcal S^D_{\bm \mu}$ reduce to the  correlators $\omega_{g,n}$ on the spectral curve $\mathcal S_{\bm \mu}$ of \cref{sec:spectral curve} (after specializing the  $\bm \mu$) upon setting $\epsilon = 0$.
\begin{lemma}\label{lem:limit}
    Set $(\mu_0,\mu_{a_1},\mu_{a_2})=(-1-D,0,0)$. Then, for every $g\in\frac12\mathbb{Z}_{\geq0}$ and $n\in\mathbb{Z}_{\geq0}$, we have
    \begin{equation*}
        \omega_{g,n+1}^{D}\Big|_{\epsilon=0}=\omega_{g,n+1}.
    \end{equation*} 
\end{lemma}
\begin{proof}
    By the construction of \cite{BonzomChapuyCharbonnierGarcia-Failde2024}, we know that $X|_{\epsilon=0}=x$ and $Y|_{\epsilon=0}=y$ (and $\tilde Y|_{\epsilon=0}=\tilde y$ ), and in particular, $\omega_{0,1}^{D}|_{\epsilon=0}=\omega_{0,1}$. It is also clear that $\omega_{0,2}^{D}|_{\epsilon=0}=\omega_{0,2}$ thanks to the relation \eqref{eq:BsigmaB}. For the above choice of ${\bm \mu}$, we have
    \begin{equation*}
         \omega_{\frac12,1}^{D}:=\frac{\bb}{2}\left(-\frac{d\tilde Y}{\tilde Y}-(1+D)\eta^0+\sum_{i=1}^D\eta^{b_i}\right).
    \end{equation*}
   Since each $b_i=\mathcal{O}(\epsilon)$ as $\epsilon\to0$, we have $\eta^{b_i}|_{\epsilon=0}=\eta^0$, and thus $\omega_{\frac12,1}^{D}|_{\epsilon=0}=\omega_{\frac12,1}$. 

   We proceed by induction on $(2g-2+n)$. Assume that the statement holds for all $(g',n')$ with $(2g'-2+n')< (2g-2+n)$. Then consider the refined topological recursion formula for  $\omega_{g,n}$ and consider the contour $C_+$ in the  formula. Recall that $\mathcal P_+ = \left\{ 0,a_1,a_2,b_1,\ldots,b_D\right\}$. For sufficiently small $\epsilon$, the elements $b_i \in \mathcal P_+$ for $i \in [D]$ are  close to the origin, while the elements $a_1,a_2$ are  close to $-u_1,-\frac{u_1u_2}{u_1+u_2+v}$ respectively.  Thus we can choose $C_+$ to be a contour (independent of $\epsilon$) such that  all the points in $\mathcal{P}_+$ remain inside $C_+$ in the limit $\epsilon\to0$. This implies that we can commute the limit $\epsilon\to0$ through the contour integral along $C_+$. Finally, we use the induction hypothesis to conclude the proof.
\end{proof}

 We are now ready to prove the main theorem of this section. Recall that the differentials $\phi^D_{g,n} $ were defined on the $n$-th product of formal disks centered around $X = \infty$. We view the $\phi^D_{g,n} $ as defined on a formal neighborhood of $z = 0$ (where $X(z) = \infty$) on the refined spectral curve $\mathcal S_{\bm \mu}$ by identifying this formal neighborhood with the aforementioned formal disk, via the identification of $X$ with the function $X(z)$ on the spectral curve.
\begin{theorem}\label{thm:w/}
Set $(\mu_0,\mu_{a_1},\mu_{a_2})=(-1-D,0,0)$. For every $(g,n) \in \frac{1}{2} \mathbb Z_{\geq 0} \times \mathbb Z_{\geq1}$, the refined topological recursion correlators  $\omega^D_{g,n}$ on the refined spectral curve $\mathcal S_{\bm \mu}$ are the analytic continuations of the shifted generating functions $\phi^D_{g,n}$ to $\Sigma^n$ under the identification  $X=X(z)$ (up to an explicit shift for $\omega_{0,1}$ and $\omega_{\frac{1}{2},1}$ given below).  In particular, 
\begin{multline}\label{eq:intfacesthm}
	\omega_{g,n}^D(z_1,\dots, z_n) - \delta_{g,0}\delta_{n,1} \epsilon V'(X(z_1))dX(z_1) - \delta_{g,\frac{1}{2}}\delta_{n,1}  \left(\frac{\bb dX(z_1)  }{2(t+X(z_1))}+\frac{\bb dX(z_1) }{2X(z_1)}\right) \\
	= \sum_{\mu_1,\dots, \mu_n \geq 1} F^D_{g,n}[\mu_1,\dots, \mu_n;\epsilon] \prod_{i=1}^n \frac{d X(z_i)}{X(z_i)^{\mu_i +1}},
\end{multline} as series expansions near  $z_i = 0$.
\end{theorem}
\begin{proof}
	For the proof, we define $\mathcal{F}_{g,n}^{D}[k_1,\ldots,k_n;\epsilon]$ as the expansion coefficients of $\omega^D_{g,n}(z_1,\ldots,z_n)$ near $z_i = 0$:
	\begin{equation*}
		\mathcal{F}^{D}_{g,n}[k_1,\ldots,k_n;\epsilon]:=(-1)^n\Res_{z_1=0}\cdots\Res_{z_n=0}\left(\prod_{i=1}^nX(z_i)^{k_i}\right)\omega_{g,n}^{D}(z_{[n]}),
	\end{equation*} 
  where $k_1,\ldots, k_n \geq 1$.  Let us consider the $\epsilon$-dependence of the $\mathcal{F}^{D}_{g,n}$. For any $m\geq 1$, we have
\begin{align*}
    \frac{\partial^m}{\partial\epsilon^m}\mathcal{F}_{g,n}^{D}[k_1,\ldots,k_n;\epsilon]&= \frac{\partial^{m-1}}{\partial\epsilon^{m-1}} \left((-1)^n\Res_{z_1=0} \cdots \Res_{z_{n}=0} \left(\prod_{i=1}^n X(z_i)^{k_i}\right) \delta_\epsilon^{(n)}  \omega_{g,n}^{D}(z_{[n]}) \right) \\
    &= (-1)^{n+1}\frac{\partial^{m-1}}{\partial\epsilon^{m-1}} \left(\Res_{z_1=0} \cdots \Res_{z_{n}=0} \Res_{z_{n+1} = 0} \left(\prod_{i=1}^n X(z_i)^{k_i}\right) V(X(z_{n+1})) \omega_{g,n+1}^{D}(z_{[n+1]}) \right) \\
    &=(-1)^{n+m}\Res_{z_1=0}\cdots\Res_{z_{n+m}=0}\left(\prod_{i=1}^nX(z_i)^{k_i}\right)\left(\prod_{j=n+1}^{n+m}V(X(z_{j}))\right)\; \omega_{g,n+m}^{D}(z_{[n+m]}),
\end{align*}where the first equality follows from \cref{lem:alpha-derivative}. To get the second equality, we note that the refined spectral curve $\mathcal{S}_{\bm \mu}^{D}$ satisfies the refined deformation condition, and use the variational formula of  \cref{thm:variation}. Repeating the process $m$ times, we arrive at the last line.

Applying a similar argument to  \cref{lem:limit}, the limit $\epsilon\to0$ and the above residues commute, and as a consequence we get
\begin{align*}
    \frac{\partial^m}{\partial\epsilon^m}\mathcal{F}_{g,n}^{D}[k_1,\ldots,k_n;\epsilon]\bigg|_{\epsilon=0} &=(-1)^{n+m} \Res_{z_1=0}\cdots\Res_{z_{n+m}=0} \left(\prod_{i=1}^nx(z_i)^{k_i}\right) \left(\prod_{j=n+1}^{n+m}V(x(z_{j}))\right)\; \omega_{g,n+m}(z_{[n+m]}) \\
     &= F^{D,m}_{g,n}\left[k_1,\ldots,k_n;\epsilon\right],
\end{align*} where we recall the definition of the $F^{D,m}_{g,n}$ from \eqref{eq:FgnD} to get the last line. When $m=0$, a slightly modified equation holds   due to \cref{lem:limit} and \cref{thm:w/o} -- the only change is the extra term for $(g,n) = (\frac{1}{2},1)$ that appears in \cref{thm:w/o}. Thus, we get  equation \eqref{eq:intfacesthm} claimed in the  theorem, and the statement regarding  analytic continuation is an immediate consequence.
\end{proof}

\subsection{Other rational weights}\label{sec:other weights} The strategy used to prove \cref{thm:w/o} and \cref{thm:w/} can be repeated verbatim to obtain  analogous versions  for the rational weights
\begin{equation}\label{eq:otherweights}
	G(z) = (u_1+z)(u_2+z), \quad  \frac{1}{v-z}, \quad \text{ or } \quad \frac{u_1+z}{v-z},
\end{equation}
which can be obtained  as limiting cases of the weight considered so far.  In the case without internal faces, the corresponding unrefined spectral curves are given by choosing $x(z) = t \frac{G(z)}{z} $ and $y(z) = \frac{z}{x(z)}$. As for the refinement, we choose
\[
\mathcal P_+  = \left\{0,-u_1\right\}, \quad \left\{0\right\}, \quad \text{or} \quad \left\{0,-u_1\right\},
\] respectively. In  all the above cases, we choose $\mu_a = 0$ for all $a \in \mathcal P_+$ except 
$\mu_0 = -1  $ to complete the definition of the refined spectral curve, which we denote $\mathcal S^G_{\bm \mu}$ for clarity. Then, we have the following theorem, which is analogous to \cref{thm:w/o}.
\begin{theorem}\label{thm:w/oother}
	For any weight $G$ in \eqref{eq:otherweights} and every $(g,n) \in \frac12\mathbb{Z}_{\geq0} \times \mathbb{Z}_{\geq1}$, the refined TR correlators $\omega_{g,n}(z_1, \ldots,z_n)$ on $ \mathcal S^G_{\bm \mu}$ are the analytic continuations of the  generating functions $\phi_{g,n}( x_1,\ldots,x_n)$ to $\Sigma^{n}$ under the identification $x_i = x(z_i)$ (up to the explicit shift for $(g,n) = (\frac{1}{2},1)$ below). In particular, as a series expansion near $z_i = 0$ (where $x(z_i) = \infty$), we have 
	\begin{equation*}
		\omega_{g,n}(z_1,\ldots, z_n)  - \delta_{g,\frac{1}{2}}\delta_{n,1} \bb \frac{dx(z)}{2 x(z)} S = \sum_{\mu_1,\ldots,\mu_n\geq1}F_{g,n}[\mu_1,\ldots,\mu_n]  \prod_{i=1}^n \frac{dx (z_i)}{x(z_i)^{\mu_i+1}},
	\end{equation*} where $S = 1$ for $G(z) = (u_1+z)(u_2+z)$ and $ S = 2$ otherwise.
\end{theorem}
The case of internal faces is similar. The unrefined spectral curves can be obtained directly from \cite{BonzomChapuyCharbonnierGarcia-Failde2024} or by taking an appropriate limit of the $X(z)$ and $Y(z)$ defined in \eqref{eq:Xdef}. By abuse of notation, we also denote the limiting functions by $X$ and $Y$. As for the refinement, we choose
\[
\mathcal P_+  = \left\{0,a_1, b_1,\cdots, b_D\right\}, \quad \left\{0, b_1,\cdots, b_D\right\}, \quad \text{or} \quad \left\{0,a_1,b_1,\cdots, b_D\right\},
\] for the weights $G$ in \eqref{eq:otherweights} respectively. Here we choose $a_1$ such that $a_1 = -u_1 + O(\epsilon)$ and $b_i = O(\epsilon)$ as $\epsilon \to 0$ as in \cref{def:refinedSCint}. We also set $\mu_0 = -1-D$, $\mu_{a_1} = 0$ and $\mu_{b_i} = 1$ for all $i \in [D]$. We denote this refined spectral curve by $ \mathcal S^{G,D}_{\bm \mu}$ for clarity. Then the following theorem which is analogous to \cref{thm:w/} holds.

\begin{theorem}\label{thm:w/other}
	For any weight $G$ in \eqref{eq:otherweights} and  every $(g,n) \in \frac{1}{2} \mathbb Z_{\geq 0} \times \mathbb Z_{\geq1}$, the refined topological recursion correlators  $\omega^D_{g,n}$ on the refined spectral curve $\mathcal S^{G,D}_{\bm \mu}$ are the analytic continuations of the shifted generating functions $\phi^D_{g,n}$ to $\Sigma^n$ under the identification  $X=X(z)$ (up to an explicit shift for $\omega_{0,1}$ and $\omega_{\frac{1}{2},1}$ given below).  In particular, 
	\begin{multline*}
		\omega_{g,n}^D(z_1,\cdots, z_n)  - \delta_{g,0}\delta_{n,1} \epsilon V'(X(z_1))dX(z_1)- \delta_{g,\frac{1}{2}}\delta_{n,1} \bb \frac{dX(z_1)}{2 X(z_1)} S\\
		= \sum_{\mu_1,\cdots, \mu_n \geq 1} F^D_{g,n}[\mu_1,\cdots, \mu_n;\epsilon] \prod_{i=1}^n \frac{d X(z_i)}{X(z_i)^{\mu_i +1}},
	\end{multline*} as series expansions near  $z_i = 0$, where $S = 1$ for $G(z) = (u_1+z)(u_2+z)$ and $ S = 2$ otherwise.
\end{theorem}

\section{Applications}
\label{sec:applications}

In this section, we will discuss how $F^G_{g,n}$ appears in various areas such as random matrix theory, representation theory, and combinatorics, where their computation is of great importance. Applying the results of the previous sections, we show that they can be computed by refined topological recursion. To avoid notational conflict in this section, we use
\begin{itemize}
\item $\omega_{g,n}^{V,D}$ for the refined topological recursion correlators, 
\item $\phi_{g,n}^{V,D}$ for colored monotone Hurwitz numbers as before, and 
\item $W_{g,n}^{V,D}$ for $\beta$-ensembles which we will introduce below,
\end{itemize}
where $V$ indicates a type of model and $D$ denotes the degree of internal faces similar to Theorem \ref{thm:w/}. The main purpose of this section is to show that $\omega_{g,n}^{V,D}=\phi_{g,n}^{V,D}=W_{g,n}^{V,D}$ in several applications.

\subsection{$\beta$-Ensemble}
\label{subsec:BetaEns}

Let $\beta >0$, and $-\infty \leq a < b \leq
\infty$. A \textit{$\beta$-ensemble}, also known in statistical mechanics
as a 1d \emph{log-gas system}, of $N$ particles in a potential $V$ at
inverse temperature $\beta$ (see~\cite{Forrester2010}) is the
probability measure $\mu_{N;\beta}^V$ on $\R^N$
defined as:
\begin{equation}
  \label{eq:BetaEns}
  d\mu_{N;\beta}^V := \frac{1}{Z_{\beta, N, V}}\prod_{i=1}^N  d \lambda_i\, \exp\left(-\frac{N\beta}{2}V(\lambda_i)\right)\mathbf{1}_{[a,b]}(\lambda_i)\prod_{1\leq i<j \leq N}
  |\lambda_i-\lambda_j|^\beta,
\end{equation}
where
\[ Z_{\beta, N, V} := \int_{\R^N}d\mu_{N;\beta}^V\]
is the normalization constant. For $\beta = 1,2,4$ this measure
coincides with the measure induced on the spectrum of matrices $M$ by
the probability distribution $dM \exp(-\Tr V(M))$ on the space of
Symmetric/Hermitian/Symplectic $N\times N$ matrices,
respectively. Studying the so-called $1/N$-expansion of these models
is an important problem in random matrix theory. For general $\beta > 0$
such an expansion was studied under
some technical analytic assumptions on the potential $V$
in~\cite{BorotGuionnet2013}. There are three \emph{classical
  cases} of the potential $V$ that are of special interest:
\begin{enumerate}
\item[{\bf (G$\beta$E)}]
  $V(x) = \frac{x^2}{2}, a = -\infty, b = \infty$, known as the
    \emph{Gaussian $\beta$ Ensemble}, \label{GaussianPot}
    \item[{\bf (J$\beta$E)}] $V(x) =
      \frac{1}{N}\left(\frac{2}{\beta}-c\right)\log(x)+\frac{1}{N}\left(\frac{2}{\beta}-d\right)\log(1-x), a = 0, b
      = 1, c,d>0$, known as the
      \emph{Jacobi $\beta$ Ensemble}, \label{JacobiPot}
      \item[{\bf (L$\beta$E)}] $V(x) =
      x+\frac{1}{N}\left(\frac{2}{\beta}-c\right)\log(x), a = 0, b
      = \infty, c>0$, known as  the \emph{Laguerre $\beta$ Ensemble}. \label{LaguerrePot}
    \end{enumerate}

For arbitrary $\beta>0$ the random tridiagonal matrix model whose spectrum
is distributed by $d\mu_{N;\beta}^V $ was constructed
by Dumitriu and Edelman in~\cite{DumitriuEdelman2002} for these three
special cases. Moreover, based on the result of
Okounkov~\cite{Okounkov1997} La Croix has proven that the $1/N$
expansion of G$\beta$E (that exists in many other cases by~\cite{BorotGuionnet2013}) has a topological expansion in terms of
$\bb$-Hurwitz numbers, with the identification of parameters $\beta=\frac{1}{\alpha}$ (see \cref{sec:b-Hurwitz}). It was recently extended to the cases of
J$\beta$E and L$\beta$E by Ruzza~\cite{Ruzza2023} who used the
result of Kadell~\cite{Kadell1997} to show that the
appropriately parametrized $1/N$ expansion of these models (studied
before for instance in~\cite{ForresterRahmanWitte2017}) also has a topological expansion in terms of
weighted $\bb$-Hurwitz numbers. We further explore this connection and
we show that the $1/N$ expansion of the correlators can be computed by
 RTR. We note that  \cite{ChekhovEynard2006a} first discussed a recursive structure in $\beta$-ensembles. Although their approach was inspiring, some subtleties were not addressed (e.g. the pole structure, choice of $\bm \mu$, choice of the contour $C_\pm$ in particular for multidifferentials, variational formula, and deformation condition) which are carefully addressed in \cite{KidwaiOsuga2023,Osuga2024,Osuga2024a}.

    The \emph{connected correlators} $\left\langle \sum_{k=1}^N\lambda_k^{k_1},\dots, \sum_{k=1}^N\lambda_k^{k_n}\right\rangle_{\mu_{N;\beta}^V}^\circ$ also known as the \emph{cumulants} are defined
    by the following combinatorial formula:
      \[ \left\langle f_1,\dots,
        f_n\right\rangle_{\mu_{N;\beta}^V}^\circ := \sum_{\pi \in
        \Pi_{\{1,\dots,n\}}}(-1)^{|\pi|-1}(|\pi|-1)!\prod_{B \in
        \pi}\int_{\R^N}\prod_{b \in
        B}f_b(\lambda_1,\dots,\lambda_N)d\mu_{N;\beta}^V,\]
    where we sum over set-partitions of the set $\{1,\dots,n\}$. For
    example:
        \begin{align*}
      \left\langle f_1\right\rangle_{\mu_{N;\beta}^V}^\circ &=\int_{\R^N}f_1d\mu_{N;\beta}^V,\\
      \left\langle f_1,f_2\right\rangle_{\mu_{N;\beta}^V}^\circ
                                                            &=\int_{\R^N}f_1f_2d\mu_{N;\beta}^V
                                                              -
                                                              \int_{\R^N}f_1d\mu_{N;\beta}^V\cdot
                                                              \int_{\R^N}f_2d\mu_{N;\beta}^V,
                                                                   \\
      \left\langle f_1,f_2,f_3\right\rangle_{\mu_{N;\beta}^V}^\circ
                                                            &=\int_{\R^N}f_1f_2f_3d\mu_{N;\beta}^V
                                                              -
                                                              \int_{\R^N}f_1d\mu_{N;\beta}^V\cdot \int_{\R^N}f_2f_3d\mu_{N;\beta}^V -
                                                              \int_{\R^N}f_2d\mu_{N;\beta}^V\cdot
                                                              \int_{\R^N}f_1f_3d\mu_{N;\beta}^V
          \\
          &-
                                                              \int_{\R^N}f_3d\mu_{N;\beta}^V\cdot \int_{\R^N}f_1f_2d\mu_{N;\beta}^V+2 \int_{\R^N}f_1d\mu_{N;\beta}^V\cdot \int_{\R^N}f_2d\mu_{N;\beta}^V\cdot \int_{\R^N}f_3d\mu_{N;\beta}^V.
        \end{align*}
        It is convenient to assemble them into a generating function:
    \begin{equation*}
      W^V_n(x_1,\dots,x_n) := \sum_{k_1,\dots,k_n \geq 1}\prod_{i=1}^n
                           \frac{dx_i}{x_i^{k_i+1}}\left\langle \sum_{k=1}^N(t\lambda_k)^{k_1},\dots, \sum_{k=1}^N(t\lambda_k)^{k_n}\right\rangle_{\mu_{N;\beta}^V}^\circ,
                         \end{equation*}
                         which is a formal power series in $t$.

The first main application of our results proves that
$W^V_n(x_1,\dots,x_n)$ for the three classical cases of $\beta$
ensembles can be computed by the refined topological recursion.

    \begin{theorem}\label{thm:ensemble}
      Let $\bb = \sqrt{\frac{\beta}{2}}-\sqrt{\frac{2}{\beta}}$, and
      let $(V(x),\mathcal{S}_{{\bm \mu}})$ be the following pairs:
\begin{enumerate}
  \item[{\bf J$\beta$E:}] $V(x)$ is as in  {\bf (J$\beta$E)} with $c =
      N\cdot(\gamma-1)+1, d = N\cdot(\delta-1)+1$, where
      $\gamma,\delta \in \C$ and
      $\mathcal{S}_{{\bm \mu}}$ is defined by
      \begin{equation*} x(z) = t\frac{(1+z)(\gamma+z)}{z(\delta+\gamma+z)}, \quad y(z)
      	= \frac{z}{x(z)}, \quad 
      	\mathcal{P}_+=\{0,-1,\frac{\gamma}{1-\delta}\},
      	\quad{\bm\mu}=\{-1,0,0\};
      \end{equation*}
  \item[{\bf L$\beta$E:}] $V(x)$ is as in  {\bf (L$\beta$E)} with $c =
      N\cdot(\gamma-1)+1$ and $\mathcal{S}_{{\bm \mu}}$ is given by
\[ x(z) = t\frac{(1+z)(\gamma+z)}{z}, \quad y(z)
      = \frac{z}{x(z)}, \quad
      \mathcal{P}_+=\{0,-1\},
      \quad{\bm\mu}=\{-1,0\};\]
      \item[{\bf G$\beta$E:}] $V(x)$ is as in  {\bf (G$\beta$E)} and
        $\mathcal{S}_{{\bm \mu}} =
        \mathcal{S}^{\text{G$\beta$E}}_{{\bm \mu}}$ from
        \cref{sec:Appendix} (see~\eqref{eq:SpectralGbetaE}-\eqref{eq:RSpectralGbetaE}).
  \end{enumerate}
  Then $W^V_n(x_1,\dots,x_n)$ can be computed by  refined topological recursion. We have
  \begin{multline*}
    \sqrt{\frac{\beta}{2}}^nW^V_n(x(z_1),\dots,x(z_n))=
    \sum_{g\in\frac12\mathbb{Z}_{\geq0}}\left(\sqrt{\frac{\beta}{2}}N\right)^{2-2g-n}\omega_{g,n}(z_1,\ldots,
    z_n)   \\
    - \begin{cases}
      \delta_{g,\frac{1}{2}}\delta_{n,1}  \left(\frac{\bb
          v}{2(x(z_1)-t)}+\frac{\bb}{2x(z_1)}\right)dx(z_1) &\text{
        for {\bf J$\beta$E}}\\
      \delta_{g,\frac{1}{2}}\delta_{n,1} \frac{\bb}{2x(z_1)}dx(z_1) &\text{
        for {\bf L$\beta$E}}
      \end{cases}
  \end{multline*}
as a series expansion near $z = 0$ (where $x(z) = \infty$). More precisely, for a fixed power of $t$, the LHS is an analytic function of $N$ whose $1/N$ expansion is given by the RHS.
\end{theorem}

\begin{proof}
In the case of  {\bf
      (G$\beta$E)} 
    \cite{Okounkov1997} proved that
    \[ \left\langle \sum_{k=1}^N(t\lambda_k)^{k_1},\dots, \sum_{k=1}^N(t\lambda_k)^{k_n}\right\rangle_{\mu_{N;\beta}^V}^\circ
  = \sqrt{\frac{2}{\beta}}^n\sum_{g\in\frac12\mathbb{Z}_{\geq0}}\left(\sqrt{\frac{2}{\beta}}\frac{1}{N}\right)^{2g-2+n}\;
  F^{\text{G$\beta$E}}_{g,n} [k_1,\ldots,k_n],\]
where $F^{\text{G$\beta$E}}_{g,n}$ are defined by
\eqref{eq:TauGBE} in \cref{sec:Appendix}, so  \cref{thm:RTRGbetaE} completes the proof in
that case. Likewise, for the cases {\bf J$\beta$E} and {\bf L$\beta$E} 
\cite{Ruzza2023} proved that 
\[ \left\langle \sum_{k=1}^N(t\lambda_k)^{k_1},\dots, \sum_{k=1}^N(t\lambda_k)^{k_n}\right\rangle_{\mu_{N;\beta}^V}^\circ
  = \sqrt{\frac{2}{\beta}}^n\sum_{g\in\frac12\mathbb{Z}_{\geq0}}\left(\sqrt{\frac{2}{\beta}}\frac{1}{N}\right)^{2g-2+n}\;
  F^G_{g,n} [k_1,\ldots,k_n]\]
for $G = \frac{(1+z)(\gamma+z)}{\gamma+\delta+z}$, and $G =
(1+z)(\gamma+z)$, respectively. Thus, the statement for  {\bf J$\beta$E}
 follows directly from \cref{thm:w/o}, while the statement for {\bf L$\beta$E} follows from \cref{thm:w/oother}.
\end{proof}

\subsection{$\beta$ Br\'ezin--Gross--Witten
  integral and $\bb$-monotone Hurwitz numbers}

When $G = \frac{1}{v-z}$, the tau
function $\tau^{(\bb)}_G$ interpolates between the partition functions for the unitary
(at $\bb =
0$) and orthogonal (at $\bb =\sqrt{2}-\sqrt{2}^{-1}$) Br\'ezin--Gross--Witten
integrals, so that it can be interpreted as their natural $\beta$
deformation whose topological expansion was found in terms of
$\bb$-monotone Hurwitz maps (described in \cref{theo:Combi};
see~\cite{BonzomChapuyDolega2023} and references therein). Thus, \cref{thm:w/oother} shows that $\bb$-monotone Hurwitz numbers are  computed by refined topological recursion.

It is worth noting that a slight variant of $\tau^{(\bb)}_G$ is known to be the Gaiotto state of $4d$ $\mathcal{N} =2$ pure
$\SU(2)$ gauge theory in arbitrary $\Omega$-background \cite{AldayGaiottoTachikawa2010, MaulikOkounkov2019, SchiffmannVasserot2013} (see~\cite[Remark
4.9]{ChidambaramDolegaOsuga2024}). Thus, the Gaiotto state can also be computed by refined topological recursion, but this study is beyond the scope of this paper.

\subsection{Maps and bipartite maps with internal faces on
  arbitrary surface}
\label{subsec:MapsInternal}

The main motivation to study the extension to internal
faces  in \cref{sec:Internal} comes from the fact that the associated $\omega_{g,n}$ produces generating
functions of special interest in enumerative combinatorics and quantum
gravity, as we will explain below.

Recall that $ \phi_{g,n}^D(X_1,\ldots,X_n)$ defined
in~\cref{sec:Internal} is the generating series for  colored monotone Hurwitz
numbers of genus $g$ with $n$ marked boundaries and with internal faces of degree $i$ weighted by
$\epsilon p_i$. It was proved in~\cite{ChapuyDolega2022} that
for two special cases of $G$-weighted $\bb$-Hurwitz numbers the
function $  \phi_{g,n}^D(X_1,\ldots,X_n)$ posses a different, particularly
simple, combinatorial interpretation that turns into two classical generating functions for
$\bb=0$ and $\bb = \sqrt{2}^{-1}-\sqrt{2}$.

Recall that a map is a 2-cell embedding of a multigraph
on a compact surface (orientable or not). We say that a map $M$  is
bipartite if one can partition its vertex set $V(M) = V_\circ(M) \cup
V_\bullet(M)$ into disjoint subsets such that no edge links two vertices
of the same subset. The degree of a face $f$ of a map $M$ is the number of corners inside $f$ that are adjacent to $V_\circ(M)$. By convention, we set $V(M) =
V_\circ(M)$ if we work with ordinary (not bipartite)
maps. We are interested in enumerating maps with $n$ marked faces of
degrees $k_1,\dots,k_n$, that we will call \emph{boundaries} and arbitrarily many faces of degree at most
$D$ that we will call \emph{internal faces},  as in
\cref{subsec:InternalFaces}. As in \cref{subsec:InternalFaces}, each boundary has a
distinguished corner that ``marks'' this face as a
boundary. See \cref{fig:TorusInternal} for an illustration of a bipartite map on the torus with internal faces of degree at most $3$.

\begin{figure}[h]
	\centering
	\includegraphics[width=0.8\textwidth
	]{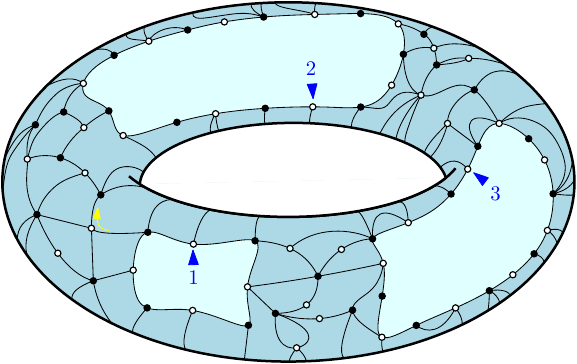}
	\caption{A bipartite map on the torus with three boundaries of
		degrees $4,9,9$, and with internal faces of degree at most
		$3$, so that it contributes to the coefficient of
		$[X_1^{-5}X_2^{-10}X_3^{-10}]$ in $ \phi_{1,3}^{{\rm
				bip},3}$. The marked corners are indicated by blue arrows,
		and the root is indicated by the yellow oriented arrow.} 
	\label{fig:TorusInternal}
\end{figure} 

Finally, we will be interested in generating series for
\emph{rooted} maps, which means that maps will have a distinguished
and oriented corner called the \emph{root}. This can be any corner in $M$,
including the marked corners lying on boundaries. For a rooted map $M$, one can use the root to define an order on the
set of edges. Then one can define the measure of non-orientability of
$M$ by
successively removing the edges from $M$ and collecting a
        weight $1$ or $-\sqrt{\alpha}\bb$ as in \cref{def:MON}, so
        that the final weight is of the form $(-\sqrt{\alpha}\bb)^{\rho(M)}$ for
                      some non-negative integer
                      $\rho(M)$. We will not need this specific
                      algorithm here, so we refer the reader to~\cite[Section
                      3]{ChapuyDolega2022} for more details. In the special cases of $\bb=0$ and $\bb =
\sqrt{2}^{-1}-\sqrt{2}$, the role of the root is not
important any more, and instead of counting rooted maps
weighted by the inverse of the number of its edges, we could equivalently count unrooted maps
weighted by the inverse of the size of their automorphism group.

\begin{remark}
  \label{rem:NonGenericB}
  One important property of this measure is that for $\bb = \sqrt{2}^{-1}-\sqrt{2}$ the weight $(-\sqrt{\alpha}\bb)^{\rho(M)} \equiv 1$ for all maps, and for $\bb = 0$ the weight $(-\sqrt{\alpha}\bb)^{\rho(M)} = 1$ only for orientable maps $M$ and vanishes for non-orientable maps.
\end{remark}

Define the generating series of all (orientable or non-orientable) rooted maps of genus $g$ with $n$ marked boundaries of degree $k_1,\dots,k_n$, and internal faces of degree at most $D$ weighted by the variable $\epsilon p_{\deg(f)}$ as
\begin{equation*}
  \phi_{g,n}^{{\rm maps},D}(X_1,\ldots,X_n) := \alpha^{-g} \sum_{k_1,\dots,k_n \geq 1}\prod_{i=1}^n\frac{dX_i}{X_i^{k_i+1}}\sum_{M}(-\sqrt{\alpha}\bb)^{\rho(M)}\frac{t^{2|E(M)|}}{2|E(M)|} u^{|V(M)|}\prod_{f \in F_i(M)}(\epsilon\cdot p_{\deg(f)}).
\end{equation*}
Similarly, for bipartite maps, define $\phi_{g,n}^{{\rm bip},D} $ as
\begin{equation*}
  \phi_{g,n}^{{\rm bip},D} (X_1,\ldots,X_n) := \alpha^{-g}\sum_{k_1,\dots,k_n \geq 1}\prod_{i=1}^n\frac{dX_i}{X_i^{k_i+1}}\sum_{M}(-\sqrt{\alpha}\bb)^{\rho(M)}\frac{t^{|E(M)|}}{|E(M)|}u_\circ^{|V_\circ(M)|}u_\bullet^{|V_\bullet(M)|}\prod_{f \in F_i(M)}(\epsilon\cdot p_{\deg(f)})
\end{equation*} 
to be the generating series of all (orientable or non-orientable) rooted bipartite maps of genus $g$ with $n$ marked boundaries of degree $k_1,\dots,k_n$, and internal faces of degree at most $D$ weighted by the variable $\epsilon p_{\deg(f)}$.

Then the main result of this section is that these generating functions can be computed by refined topological recursion.
\begin{theorem}  \label{theo:GenSerMaps} For maps and bipartite maps, we have
  \begin{itemize}
    \item{\emph{Maps:}} We have
    \begin{equation*}
      \phi_{g,n}^{{\rm maps},D} (X(z_1),\ldots,X(z_n))= \omega_{g,n}^{{\text{G$\beta$E}},D}(z_1,\ldots,z_n)- \delta_{g,0}\delta_{n,1} \epsilon V'(X(z_1))dX(z_1)
    \end{equation*}
    where $\omega_{g,n}^{{\text{G$\beta$E}},D}$ are the refined topological recursion correlators on the spectral curve given in \cref{sub:InternalGBE}.
    \item{\emph{Bipartite maps:}} We have
    \begin{equation*}
      \phi_{g,n}^{{\rm bip},D} (X(z_1),\ldots,X(z_n)) = \omega_{g,n}^{{\rm bip},D} (z_1,\ldots,z_n)- \delta_{g,0}\delta_{n,1} \epsilon V'(X(z_1))dX(z_1)-\delta_{g,\frac12}\delta_{n,1}\frac{\bb}{2X(z_1)}dX(z_1),
    \end{equation*}
    where $\omega_{g,n}^{{\rm bip},D}$ are refined topological recursion correlators on the spectral curve $\mathcal{S}_{\bm\mu}^{D}$ defined in \cref{sec:other weights} for $G(z) = (u_\circ+z) (u_\bullet+z)$.
  \end{itemize}
\end{theorem}

\begin{proof}
  Let $F^{\text{G$\beta$E}}_{g,n} [\mu_1,\ldots,\mu_n]$ be defined by \eqref{eq:TauGBE}, and $F^{{\rm bip}}_{g,n} [\mu_1,\ldots,\mu_n] = F^{G}_{g,n} [\mu_1,\ldots,\mu_n]$ for $G(z) = (u_\circ+z) (u_\bullet+z)$.  Then, \cite{ChapuyDolega2022} proves that
  \begin{align*}
    \frac{\alpha^g}{|\Aut(\mu)|\prod_{i=1}^n\mu_i}F^{\text{G$\beta$E}}_{g,n}[\mu_1,\ldots,\mu_n]&=\frac{t^{2(k_1+\dotsm+k_n)}}{2(k_1+\dotsm+k_n)}\sum_{M}(-\sqrt{\alpha}\bb)^{\rho(\mathcal{M})}u^{|V(M)|},\\
    \frac{\alpha^g}{|\Aut(\mu)|\prod_{i=1}^n\mu_i}F^{{\rm bip}}_{g,n}[\mu_1,\ldots,\mu_n]&=\frac{t^{k_1+\dotsm+k_n}}{k_1+\dotsm+k_n}\sum_{M}(-\sqrt{\alpha}\bb)^{\rho(\mathcal{M})}u_\circ^{|V_\circ(M)|}u_\bullet^{|V_\bullet(M)|},
  \end{align*}
  where the first sum is over all rooted maps of genus $g$ with $n$ faces of degrees $\mu_1,\dots,\mu_n$, and the second sum is over all rooted bipartite maps of genus $g$ with $n$ faces of degrees $\mu_1,\dots,\mu_n$. Define $F^{\text{G$\beta$E},D}_{g,n}[k_1,\ldots,k_n;\epsilon]$ by \eqref{eq:InternalGBE} and $F^{{\rm bip},D}_{g,n}[k_1,\ldots,k_n;\epsilon]$ by \eqref{eq:HurwitzWithBoundaries} for $G(z) = (u_\circ+z) (u_\bullet+z)$. Then, using the same arguments as in the proof of \cref{prop:HurwitzWithBoundaries} one can show that for both maps and bipartite maps, we have
  \begin{align*} 
    \phi_{g,n}^{{\rm maps (bip)},D} (X_1,\ldots,X_n) &= \sum_{k_1,\dots,k_n \geq 1}\prod_{i=1}^n\frac{dX_i}{X_i^{k_i+1}} F^{\text{G$\beta$E (bip)},D}_{g,n}[k_1,\ldots,k_n;\epsilon].
  \end{align*}
  Then, the first part of \cref{theo:GenSerMaps} concerning maps follows from \cref{theo:InternalGBE}, while the second part concerning bipartite maps follows from \cref{thm:w/other}.
\end{proof}

\begin{corollary}
 For $2g-2+n \geq 1$ the  formal series $\phi_{g,n}^D(X(z_1),\ldots,X(z_n))$from \cref{theo:GenSerMaps}, for both maps and bipartite maps, can be written as a rational multidifferential in $z_1,\dots,z_n$.
                        \end{corollary}

While the above corollary is a straightforward consequence of refined topological recursion, the problem of rational parametrization of generating series in enumerative combinatorics of maps has a rich history of studies employing algebraic, analytic, and bijective methods~\cite{BenderCanfield1986,BenderCanfield1991,Gao1993,BousquetJehanne2006,Chapuy2009,ChapuyMarcusSchaeffer2009,ChapuyFang2016,ChapuyDolega2017}. In the special case $\bb = 0$, which corresponds to the generating series of orientable maps/bipartite maps, the rational parametrization of generating series was known before. It was recently reproved using topological recursion in~\cite{BonzomChapuyCharbonnierGarcia-Failde2024,BychkovDuninBarkowskiKazarianShadrin2025} for a much wider class of maps than the ones studied here. See also \cite{BranahlHock2025} for an efficient way of computing bipartite maps from topological recursion in the unrefined setting

On the other hand, the rational parametrization even in the special case of $\bb=\sqrt{2}^{-1}-\sqrt{2}$, which corresponds to the generating series of non-oriented maps/bipartite maps, is a new result obtained here. In fact, while the rational parametrization manifests itself through topological recursion in both the refined and unrefined setting, the analytic structure of the correlators in the refined setting differs substantially. Indeed, the RTR correlators have poles along the ``anti-diagonals'' $z_i = \sigma(z_j) $ and at $z_i \in \sigma(\mathcal P_+)$, for any $i,j \in [n]$. 

It was noticed in the past that the structure of the generating series of non-oriented maps is not always as simple as their orientable counterparts~\cite{ArquesGiorgetti2000,DolegaLepoutre2022,BonzomChapuyDolega2022}. We plan to address this problem in  future work, especially from the point of view of a conjectural universal pattern in asymptotic enumeration that so far was explained in the orientable case by the structural properties of TR~\cite{Eynard2016}, but in the non-orientable case remains elusive.

\appendix

\section{Gaussian Beta Ensembles}
\label{sec:Appendix}

So far, we studied tau functions of weighted single $\bb$-Hurwitz numbers,
which are obtained from tau functions of weighted double $\bb$-Hurwitz numbers $\tau_G^{{(\bb)}}(\tilde{\pp}, \tilde{\qq})$ by the
specialization $\tau_G^{(\bb)} := \tau_G^{(\bb)}(\tilde{\pp},
\tilde{\qq})\big|_{q_i := \delta_{i,1}}$. This more
general tau function for double Hurwitz numbers is defined as
\begin{equation*}
    \tau_G^{(\bb)}(\tilde{\pp}, \tilde{\qq})=\sum_{d \geq 0} t^d\sum_{\lambda \vdash d} 
		\frac{J_\lambda^{(\alpha)}(\sqrt{\alpha}\tilde{\pp}) J_\lambda^{(\alpha)}(\sqrt{\alpha}\hbar^{-1}\tilde{\qq})}{j_\lambda^{(\alpha)}}\prod_{\square
                  \in \lambda}G(\hbar\cdot
  \tilde{c}_{\alpha}(\square)).
\end{equation*}

In this appendix we briefly discuss one extra case of weighted
$\bb$-Hurwitz theory that is of special interest due to its relation to
random matrices and combinatorial maps. We define
$\tau_{\text{G$\beta$E}}^{(\bb)}$ as the following specialization of $\tau_G^{(\bb)}(\tilde{\pp}, \tilde{\qq})$
\begin{align}
  \label{eq:TauGBE}
    \tau_{\text{G$\beta$E}}^{(\bb)} :=\tau_G^{(\bb)}(\tilde{\pp}, \tilde{\qq})\big|_{\tilde{q}_i = \delta_{i,2}}=\exp\left(\sum_{g\in\frac12\mathbb{Z}_{\geq0}}\sum_{n\in\mathbb{Z}_{\geq1}}\sum_{\mu_1,\ldots,\mu_n\in\mathbb{Z}_{\geq1}}\frac{\hbar^{2g-2+n}}{n!}\;
  F^{\text{G$\beta$E}}_{g,n} [\mu_1,\ldots,\mu_n] \;\frac{\tilde p_{\mu_1}}{\mu_1}\cdots
  \frac{\tilde p_{\mu_n}}{\mu_n}\right) ,
\end{align}
with the choice $G(z) = u+z$.  We use the terminology G$\beta$E, given that this model essentially coincides with the Gaussian $\beta$ ensemble considered in \cref{subsec:BetaEns}. Although $\tau_{\text{G$\beta$E}}^{(\bb)}$ is not covered by the results of \cite{ChidambaramDolegaOsuga2024}, the  constraints required to apply the methods of this paper are known.  More precisely, we have the following  Virasoro constraints\footnote{There is a sign typo in \cite[Proposition A.1]{BonzomChapuyDolega2023}; the corresponding tau functions defined by eq. (81) should be corrected by replacing the minus sign in front of $c_b$ by a plus sign.}.

\begin{proposition}[\cite{AdlervanMoerbeke2001,BonzomChapuyDolega2023}]
For all $k\in\mathbb{Z}_{\geq-1}$, the operators $L_k$ defined by
\begin{align*}
L_k= \sum_{j\geq 1}  \mathsf J_{k-j}  \mathsf J_j+
  \big(2u-\bb\hbar(k+1)\big) \mathsf J_k
  -u \mathsf J_{-1}\delta_{k,-1}+\left(u^2-u\hbar\bb\right)
  \delta_{k,0}-\frac{1}{t^2}\mathsf J_{k+2}.
\end{align*}
annihilate the tau function:
\begin{equation*}
    L_k \tau_{\text{G$\beta$E}}^{(\bb)}=0, \quad \text{ for } k \geq -1.
\end{equation*}
\end{proposition}

For any   $(g,n)\in\frac12\mathbb{Z}_{\geq0} \times \mathbb{Z}_{\geq1}$, we  consider the  generating functions 
\begin{align*}
    W_{g,n}^{{\text{G$\beta$E}}}(x_1,\ldots,x_n) := [\hbar^{2g-2+n}]\prod_{i=1}^n\nabla(x_i)\log\tau_{{\text{G$\beta$E}}}^{(\bb)}\bigg|_{\pp=0}.
\end{align*} We also consider a shifted version which is helpful for the calculations below, defined as $\tilde W_{g,n}^{{\text{G$\beta$E}}}(x_1,\ldots,x_n)  := W_{g,n}^{{\text{G$\beta$E}}}(x_1,\ldots,x_n) -\delta_{g,0}\delta_{n,1}\left(\frac{x_1}{2t^2}-\frac{u}{x_1}\right)dx_1$.
Then, we obtain the following loop equations, whose  proof is completely parallel to  \cref{prop:FLE} and hence omitted.

\begin{proposition} We have the following formal loop equations
\begin{equation} \label{formalRLEGUE}
     Q^{\tilde W}_{g,1+n}(x;x_{[n]}) = \delta_{g,0}\delta_{n,0}\left(\frac{x^2}{4t^4}-\frac{u}{t^2}\right)dx^2 - \delta_{g,\frac12}\delta_{n,0}\left(\bb\frac{1}{2t^2}\right)dx^2 + dx^2\sum_{i=1}^nd_i\left(\frac{\tilde W_{g,n}^{{\text{G$\beta$E}}}(J)}{(x-x_i)dx_i}\right), 
\end{equation}
where $d$ and $d_i$ denote the exterior derivative with respect to $x$ and $x_i$ respectively.
\end{proposition}

From these loop equations, we can  derive the refined topological recursion. Let us first consider two meromorphic functions $x,y$ on $\mathbb{P}^1_z$ as
\begin{equation}
  \label{eq:SpectralGbetaE}
    x(z)=t\frac{1+uz^2}{z},\quad y(z)=\frac{u^2z^3}{t(1+u z^2)},
\end{equation} and define $\tilde y(z) := y(z)-\left(\frac{x(z)}{2t^2}-\frac{u}{x(z)}\right)$.
The functions $x$ and $\tilde y$ satisfy an algebraic relation
\begin{equation*}
\tilde y(z)^2-\frac{x(z)^2}{4t^4}+\frac{u}{t^2}=0,
\end{equation*}
which coincides with the equation \eqref{formalRLEGUE} for $(g,n)=(0,1)$ under the  identification $x = x(z)$ and $\tilde W^{\text{G$\beta$E}}_{0,1}(x)=\tilde y(z)dx(z)$. By removing the shifts, we immediately get that $ydx$ is the analytic continuation of $W^{\text{G$\beta$E}}_{0,1}$. We define $\omega^\text{G$\beta$E}_{0,1}:=ydx$ on $\mathbb{P}^1_z$ as usual.

Similar computations show that $\omega^\text{G$\beta$E}_{0,2}(z_1,z_2):=-B(z_1,\sigma(z_2))$ is the analytic continuation of $W^\text{G$\beta$E}_{0,2}(x_1,x_2)$ where $\sigma(z) = \frac{1}{uz}$. In order to fix the refined spectral curve $\mathcal{S}_{\bm\mu}^{{\text{G$\beta$E}}}$, we choose $\mathcal{P_+}=\{0\}$ and the associated parameter $\bm\mu=\{-1\}$ to define
\begin{equation}
\label{eq:RSpectralGbetaE}
    \omega_{\frac12,1}^{{\text{G$\beta$E}}}=\frac{\bb}{2}\left(-\frac{d\tilde y}{\tilde y}-\eta^0\right).
\end{equation}
Let us denote the correlators constructed by the refined topological recursion on $\mathcal{S}_{\bm\mu}^{{\text{G$\beta$E}}}$ by $\omega_{g,n}^{{\text{G$\beta$E}}}$ . Then, the $W^\text{G$\beta$E}_{g,n}$ coincide with the RTR correlators $\omega^{\text{G$\beta$E}}_{g,n}$ (note that there is no shift for $(g,n)= (\frac{1}{2},1)$ in contrast with \cref{thm:w/o}).

\begin{theorem}
	\label{thm:RTRGbetaE}
	For $(g,n) \in\frac12\mathbb{Z}_{\geq0} \times \mathbb{Z}_{\geq1}$, the G$\beta$E correlators analytically continue to the refined TR correlators $\omega^{\text{G$\beta$E}}_{g,n}$ on the refined spectral curve $\mathcal{S}_{\bm\mu}^{{\text{G$\beta$E}}}$. As a series expansion near $z_i = 0$ (where $x(z_i) = \infty$), we have
	\[
		\omega_{g,n}(z_1,\cdots,z_n) = \sum_{\mu_1,\cdots, \mu_n \geq 1} F_{g,n}\left[\mu_1,\cdots, \mu_n\right] \prod_{i=1}^n \frac{dx(z_i)}{x(z_i)^{\mu_i+1}}.
	\]
\end{theorem}
\begin{proof}
	In place of equations \eqref{eta0} and \eqref{kernel}, in this model we have
	\begin{equation*}
		\eta^0(z)= \frac{dx(z)}{2t^2\tilde y(z)},\qquad
		\tilde y(z) dx(z) \frac{\eta^{z_1}(z)}{\tilde y(z_1) dx(z_1)}=\frac{1}{x(z)-x(z_1)}\frac{dx(z)^2}{dx(z_1)}.
	\end{equation*}
	The rest of the arguments run parallel to the proof of Theorem \ref{thm:w/o}.
\end{proof}

 \subsection{Internal faces}
 \label{sub:InternalGBE}
 We can also extend this result to the model with internal faces. The
 special case $\bb = 0$ was already treated
 in~\cite{BonzomChapuyCharbonnierGarcia-Failde2024}, where the
 authors constructed two meromorphic functions $X,Y$ that define the
 spectral curve for this model (see Section 2.1  of \textit{loc.cit.}). Using these two functions and defining a shifted version $\tilde Y$ of $Y$ that is anti-invariant under the involution $\sigma$, one can show 
 (as in \cref{lem:Y^2}) that there exists a polynomial $M(X)$ of degree $D$ such that $M(X)|_{\epsilon=0}=1$ and
\begin{equation*}
    \tilde Y^2=\frac{(X-a)(X-b)M(X)^2}{4t^4}.
\end{equation*}
We then choose $\mathcal{P}_+=\{0,b_1,\ldots,b_D\}$ such that $b_i=\mathcal{O}(\epsilon)$ as $\epsilon\to0$, and the associated parameter $\bm\mu$ is set to
\begin{equation*}
\setlength\arraycolsep{1pt}
    \begin{matrix}
\mathcal{P}_+ & = & \{& 0,& b_1,&,\ldots,& b_D&\}\\
 {\bm \mu} & = & \{&\mu_0,&1,&,\ldots,&1&\},
\end{matrix}
\end{equation*}
For any $\mu_0\in\mathbb{C}$, one can show that the refined spectral curve $\mathcal{S}_{\bm\mu}^{{\text{G$\beta$E}},D}$ satisfies the refined deformation condition with respect to $\epsilon$, similar to Proposition \ref{prop:deform}. The proof is completely parallel, except that in place of the $\kappa$ used in the proof of \cref{lem:1/21}, we should define $\kappa \left(\mathcal{P}_+\right) = \{-D-1,1,\ldots,1\}$.

To state the result, we define the  function
$F^{{\text{G$\beta$E}},D}_{g,n}\left[k_1,\ldots,k_n;\epsilon\right]$  as
\begin{equation}
	\label{eq:InternalGBE}
	F_{g,n}^{{\text{G$\beta$E}},D}\left[k_1,\ldots,k_n;\epsilon\right] := \sum_{m \geq
		0}\frac{\epsilon^m}{m!}\sum_{k_{n+1},\dots,k_{n+m}=1}^DF^{{\text{G$\beta$E}}}_{g,n}[k_1,\dots,k_n,k_{n+1},\dots,k_{n+m}]\prod_{i=n+1}^{m+n}\frac{p_{k_i}}{k_i},
\end{equation} which  allows for internal faces of degree $i  \leq D$, which are weighted by $\epsilon p_i$.
Then, by repeating the proof of Theorem \ref{thm:w/}, we get the following theorem whose proof is omitted.

\begin{theorem}
  \label{theo:InternalGBE}
    Set $\mu_0=-1-D$. For $(g,n) \in\frac12\mathbb{Z}_{\geq0} \times \mathbb{Z}_{\geq1}$, the function $F^{{\text{G$\beta$E}},D}_{g,n}$ where we allow arbitrary many internal faces can be computed by refined topological recursion. More precisely, near $ z_i= 0 $ where $X(z_i) = \infty$, we have
    \[
    \omega_{g,n}(z_1,\cdots,z_n)  - \delta_{g,0}\delta_{n,1} \epsilon V'(X(z_1))dX(z_1)= \sum_{\mu_1,\cdots, \mu_n \geq 1} F_{g,n}^{{\text{G$\beta$E}},D}\left[\mu_1,\cdots, \mu_n;\epsilon\right] \prod_{i=1}^n \frac{dX(z_i)}{X(z_i)^{\mu_i+1}}.
    \]
\end{theorem}

\bibliographystyle{amsalpha}

\bibliography{biblio2015}

\newpage

\end{document}